\documentclass[11pt]{article}
\usepackage{latexsym}
\usepackage{amsmath,color}
\usepackage{epsfig}

\font\Cp = msbm10

\newcommand{\Rrr}{\hbox{\Cp R}}
\newcommand{\Zzz}{\hbox{\Cp Z}}

\newcommand{\qed}{\mbox{$\Box$}\vspace{\baselineskip}}

\newenvironment{proof}{\noindent {\bf Proof:}}{{\qed}}
\newenvironment{proof_special}{\noindent {\bf Proof:}}{}

\newtheorem{theorem}{Theorem}[section]
\newtheorem{proposition}[theorem]{Proposition}
\newtheorem{lemma}[theorem]{Lemma}
\newtheorem{corollary}[theorem]{Corollary}

\newtheorem{example}[theorem]{Example}
\newtheorem{continuation}[theorem]{Continuation of Example}
\newtheorem{definition}[theorem]{Definition}

\makeatletter
\@addtoreset{equation}{section}
\makeatother

\newcommand{\onethingatopanother}[2]
{\genfrac{}{}{0pt}{}{#1}{#2}}

\newcommand{\hz}{\hat{0}}
\newcommand{\ho}{\hat{1}}

\newcommand{\tensor}{\otimes}

\newcommand{\ab}{\av\bv}
\newcommand{\av}{{\bf a}}
\newcommand{\bv}{{\bf b}}
\newcommand{\cd}{\cv\dv}
\newcommand{\ctd}{\cv\mbox{-}2\dv}
\newcommand{\cv}{{\bf c}}
\newcommand{\dv}{{\bf d}}

\newcommand{\id}{\operatorname{id}}
\newcommand{\wt}{\operatorname{wt}}
\newcommand{\lcm}{\operatorname{lcm}}

\newcommand{\HH}{{\mathcal H}}
\newcommand{\LL}{{\mathcal L}}
\newcommand{\Lzero}{\LL \cup \{\hz\}}
\newcommand{\Lub}{\LL_{ub}}
\newcommand{\PP}{{\mathcal P}}
\newcommand{\Pzero}{\PP \cup \{\hz\}}

\newcommand{\zab}{\Zzz\langle\av,\bv\rangle}
\newcommand{\zcd}{\Zzz\langle\cv,\dv\rangle}

\newcommand{\tensordots}{\tensor \cdots \tensor}
\newcommand{\cupdots}{\cup \cdots \cup}
\newcommand{\capdots}{\cap \cdots \cap}

\newcommand{\pair}[2]{\left\langle#1 \: \vrule \: #2\right\rangle}

\newcommand{\zub}{z_{ub}}
\newcommand{\Zub}{Z_{ub}}
\newcommand{\zt}{z_{t}}
\newcommand{\Zt}{Z_{t}}

\newcommand{\ma}{\av}
\newcommand{\mb}{\bv}
\newcommand{\maa}{\av^{2}}
\newcommand{\mab}{\av\bv}
\newcommand{\mba}{\bv\av}
\newcommand{\mbb}{\bv^{2}}
\newcommand{\maaa}{\av^{3}}
\newcommand{\maab}{\av^{2}\bv}
\newcommand{\naab}{\av\av\bv}
\newcommand{\maba}{\av\bv\av}
\newcommand{\mabb}{\av\bv^{2}}
\newcommand{\nabb}{\av\bv\bv}

\newcommand{\nbaa}{\bv\av\av}
\newcommand{\mbab}{\bv\av\bv}

\newcommand{\nbba}{\bv\bv\av}

\newcommand{\mccc}{\cv^3}
\newcommand{\mcd}{\cv \dv}
\newcommand{\mdc}{\dv \cv}

\begin{document}

\title{Affine and toric hyperplane arrangements}
\author{{\sc Richard EHRENBORG,}
        {\sc Margaret READDY}
          and
        {\sc Michael SLONE}}
\date{}
\maketitle

\begin{abstract}
We extend the Billera--Ehrenborg--Readdy
map between the intersection lattice
and face lattice of
a central hyperplane arrangement to
affine and toric hyperplane arrangements.
For arrangements on the torus,
we also generalize Zaslavsky's fundamental
results on the number of regions.
\end{abstract}

\section{Introduction}
\label{section_introduction}

Traditionally combinatorialists have studied topological objects
that are spherical, such as polytopes, or which are homeomorphic to
a wedge of spheres, such as those obtained from shellable complexes.
In this paper we break from this practice and study hyperplane 
arrangements on the $n$-dimensional torus.

It is classical that the convex hull
of a finite collection of points in Euclidean space
is a polytope
and its boundary is a sphere.
The key ingredient in this construction is convexity.
At the moment there is no natural analogue of this process
to obtain a complex whose geometric realization is
a torus.

In this paper we are taking a zonotopal approach to 
working with arrangements on the torus.
Recall that a zonotope can be defined without the notion
of convexity, that is, it is a Minkowski sum of line segments.
Dually, a central hyperplane arrangement gives rise to
a spherical cell complex.
By considering an arrangement on the torus, we are able to
obtain a subdivision whose geometric realization is indeed the torus.
We will see later in
Section~\ref{section_toric}
that
this amounts to restricting ourselves to arrangements
whose subspaces
in the Euclidean space $\Rrr^{n}$
have coefficient matrices with rational entries.
Under the quotient map
$\Rrr^{n} \longrightarrow \Rrr^{n}/\Zzz^{n} = T^{n}$
these subspaces 
are sent to subtori of the $n$-dimensional torus $T^{n}$.

Zaslavsky
initiated the modern study of hyperplane arrangements in his
fundamental treatise~\cite{Zaslavsky}.
For early work in the field,
see the references given
in Gr\"unbaum's text~\cite[Chapter 18]{Grunbaum}.
Zaslavsky showed that evaluating the characteristic
polynomial of a central hyperplane arrangement at
$-1$ gives the number of regions
in the complement of the arrangement.
For central hyperplane arrangements,
Bayer and Sturmfels~\cite{Bayer_Sturmfels} proved the flag $f$-vector
of the arrangement can be determined from the intersection
lattice;
see Theorem~\ref{theorem_Bayer_Sturmfels}.
However, their result is stated as a sum
of chains in the intersection lattice and hence
it is hard to apply.
Billera, Ehrenborg and Readdy
improved the 
Bayer--Sturmfels result by showing that it is enough
to know the flag $f$-vector of the intersection lattice
to compute the flag $f$-vector of a central arrangement.
Recall that the $\cd$-index of a regular
cell complex is an efficient tool to encode its
flag $f$-vector without linear redundancies~\cite{Bayer_Klapper}.
The Billera--Ehrenborg--Readdy theorem gives an explicit way to compute
the $\cd$-index of the arrangement,
and hence its flag $f$-vector~\cite{Billera_Ehrenborg_Readdy_om}.

We generalize Zaslavsky's theorem on the number
of regions of a hyperplane arrangement to the toric case.
Although there is no intersection lattice per se,
one works with the intersection poset.
From the Zaslavsky result we obtain a toric
version of the Bayer--Sturmfels result for
hyperplane arrangements,
that is, there is a natural poset map from
the face poset to the intersection poset,
and furthermore, the cardinality
of the inverse image of a chain under this map
is described.

As in the case of a central hyperplane 
arrangement, our toric version of the Bayer--Sturmfels result determines
the flag $f$-vector of the face poset of
a toric arrangement in terms of its intersection poset.
However, this is far from being explicit.
Using the coalgebraic techniques
from~\cite{Ehrenborg_Readdy_c},
we are able to determine the flag $f$-vector
explicitly in terms of the flag $f$-vector
of the intersection poset. 
Moreover, the 
answer is given by a $\cd$ type of polynomial.
The flag $f$-vector of a regular spherical complex
is encoded by the $\cd$-index, a non-commutative polynomial
in the variables $\cv$ and~$\dv$,
whereas
the $n$-dimensional
toric analogue
is a $\cd$-polynomial plus the $\ab$-polynomial $(\av-\bv)^{n+1}$.

Zaslavsky also showed 
that evaluating the characteristic
polynomial of an affine arrangement at
$1$ gives the number of bounded regions
in the complement of the arrangement.
Thus we return to affine arrangements in
Euclidean space with the twist that
we study the {\em unbounded} regions.
The unbounded regions form a spherical
complex.
In the case of central arrangements,
this complex is exactly what was studied
previously  
by Billera, Ehrenborg and Readdy~\cite{Billera_Ehrenborg_Readdy_om}.
For non-central arrangements,
we determine the $\cd$-index of
this complex in terms of the lattice
of unbounded intersections of the arrangement.

Interestingly, the techniques for studying toric arrangements and
the unbounded complex of 
non-central arrangements are similar.
Hence, we present these
results
in the same paper.
For example, 
the toric and non-central
analogues of the Bayer--Sturmfels theorem
only differ
by which Zaslavsky invariant is used.
The coalgebraic translations of the
two analogues
involve exactly the same argument,
and
the resulting underlying maps
$\varphi_{t}$ (in the toric case) and 
$\varphi_{ub}$ (in the non-central case)
differ only slightly
in their definitions.

We end with many open questions about subdivisions of manifolds.

\section{Preliminaries}
\label{section_preliminaries}

All the posets we will work with are
graded,
that is,
posets having a unique minimal element
$\hz$, a
unique maximal element
$\ho$,
and rank function $\rho$.
For two elements $x$ and $z$ in a graded poset $P$ such that $x \leq z$,
let $[x,z]$ denote the interval
$\{y \in P \: : \: x \leq y \leq z\}$.
Observe that the interval $[x,z]$ is itself a graded poset.
Given a graded poset $P$ of rank $n+1$
and
$S \subseteq \{1, \ldots, n\}$,
the
{\em $S$-rank-selected poset $P(S)$}
is
the poset
consisting of the elements
$P(S) = \{x \in P \: : \: \rho(x) \in S\} \cup \{\hz, \ho\}$.
The partial order of
$[x,y]$ and $P(S)$ are each inherited from that
of $P$.
For a graded poset $P$ let $P^{*}$ denote the dual
poset, that is, the poset having
the same underlying set as $P$ but with the
order relation reversed:
$x <_{P^{*}} y$ if and only if $y <_{P} x$.
For standard poset terminology,
we refer the reader
to Stanley's work~\cite{Stanley_EC_I}.

The M\"obius function $\mu(x,y)$ on a poset $P$ is
defined recursively by
$\mu(x,x) = 1$
and for elements $x, y \in P$ with $x < y$ by 
$\mu(x,y) = - \sum_{x \leq z < y} \mu(x,z)$;
see Section~3.7 in~\cite{Stanley_EC_I}.
For a graded poset~$P$ 
with minimal element $\hz$ and maximal element $\ho$
we write $\mu(P) = \mu_{P}(\hz,\ho)$.

We now review important  results about
hyperplane arrangements, the $\cd$-index
and coalgebraic techniques.
All are essential
for proving the main results of this paper.

\subsection{Hyperplane arrangements}

Let $\HH = \{H_{1}, \ldots, H_{m}\}$
be a hyperplane arrangement in $\Rrr^{n}$,
that is,
a finite collection of affine hyperplanes in 
$n$-dimensional Euclidean space $\Rrr^{n}$.
For brevity, throughout this paper we will 
often refer to a
hyperplane arrangement as an arrangement.
We call an arrangement {\em essential} if
the normal vectors to the hyperplanes in $\HH$
span $\Rrr^{n}$.
An arrangement that is not essential can be made essential by
quotienting out by the subspace $V^{\perp}$ where
$V$ is the subspace orthogonal to 
all of the hyperplanes in $\HH$.
In this paper we are only interested
in essential arrangements.

Observe that the intersection 
$\bigcap_{i=1}^{m} H_{i}$
of all of the hyperplanes in an essential arrangement is either
the empty set $\emptyset$ or a singleton point.
We call an arrangement {\em central} if the
intersection of all the hyperplanes is one point.
We may assume that this point is the origin ${\mathbf 0}$
and hence all of the hyperplanes are codimension $1$
subspaces.
If the intersection is the empty set,
we call the arrangement {\em non-central}.

The {\em intersection lattice} $\LL$
is the lattice formed by ordering all the intersections
of hyperplanes in~$\HH$ by reverse inclusion.
If the intersection
of all the hyperplanes
in a given arrangement is empty,
then we include the empty set
$\emptyset$ as the 
the maximal element in the intersection lattice.
If the arrangement is central, the
maximal element
is  $\{{\mathbf 0}\}$.
In all cases, the minimal element
of $\LL$ will be all of
$\Rrr^{n}$.

For a hyperplane arrangement $\HH$
with intersection lattice $\LL$, the
{\em characteristic polynomial} is defined by
$$   \chi(\HH; t)
   =
     \sum_{\onethingatopanother{x \in \LL}{x \neq \emptyset}}
             \mu(\hz,x) \cdot t^{\dim(x)}  ,   $$
where $\mu$ denotes the
M\"obius function.
The characteristic polynomial is a combinatorial
invariant of the arrangement. 
The fundamental result of Zaslavsky~\cite{Zaslavsky}
is that this
invariant determines the number and type of
regions.
\begin{theorem}[Zaslavsky]
For a hyperplane arrangement $\HH$ in $\Rrr^{n}$
the number of
regions is given by $(-1)^{n} \cdot \chi(\HH; t=-1)$.
Furthermore,
the number of
bounded regions is given by $(-1)^{n} \cdot \chi(\HH; t=1)$.
\label{theorem_Zaslavsky}
\end{theorem}
For a graded poset $P$,
define the two Zaslavsky invariants $Z$ and $Z_{b}$
by
\begin{eqnarray*}
      Z(P) 
  & = &
      \sum_{\hz \leq x \leq \ho} (-1)^{\rho(x)} \cdot \mu(\hz,x)  ,\\
      Z_{b}(P) 
  & = &
      (-1)^{\rho(P)} \cdot \mu(P)    .
\end{eqnarray*}
In order to work with Zaslavsky's result, we need the following
reformulation of Theorem~\ref{theorem_Zaslavsky}.
\begin{theorem}
\begin{enumerate}
\item[(i)]
For a central hyperplane arrangement the number of
regions is given by $Z(\LL)$, where $\LL$
is the intersection lattice of the arrangement.

\item[(ii)]
For a non-central hyperplane arrangement the number of
regions is given by $Z(\LL) - Z_{b}(\LL)$,
where $\LL$
is the intersection lattice of the arrangement.
The number of bounded regions is given by~$Z_{b}(\LL)$.
\end{enumerate}
\label{theorem_Zaslavsky_poset}
\end{theorem}

Given a central hyperplane arrangement $\HH$
there are two associated lattices, namely,
the 
intersection lattice $\LL$
and the 
lattice $T$ of faces of the arrangement.
The minimal element of $T$ is the empty set $\emptyset$
and the maximal element is the whole space $\Rrr^{n}$.
The lattice of faces can be seen as the face poset
of the cell complex obtained by intersecting the arrangement
$\HH$ with a sphere of radius $R$
centered at the origin.
Each hyperplane corresponds to a great circle on the sphere.
An alternative way to view the lattice of faces $T$
is that the dual lattice $T^{*}$ is the 
face lattice of the zonotope corresponding to $\HH$.

Let $\Lzero$ denote the intersection lattice
with a new minimal element $\hz$ adjoined.
Define an order- and rank-preserving map $z$
from the dual lattice $T^{*}$ to 
the augmented lattice $\Lzero$
by sending a face of the arrangement, that is, a cone in
$\Rrr^{n}$, to its affine hull.
Note that under the map $z$ the minimal
element of $T^{*}$ is mapped to the 
minimal element of $\Lzero$.
Observe that $z$ maps chains to chains.
Hence we view $z$ as a map from the set
of chains of $T^{*}$
to the set of chains of $\Lzero$.
Bayer and Sturmfels~\cite{Bayer_Sturmfels}
proved the following result about the
inverse image of a chain under the map $z$.
\begin{theorem}[Bayer--Sturmfels]
Let $\HH$ be a central hyperplane arrangement
with intersection lattice $\LL$.
Let $c = \{\hz = x_{0} < x_{1} < \cdots < x_{k} = \ho\}$
be a chain in $\Lzero$. Then the cardinality
of the inverse image of the chain $c$ 
under the map $z : T^{*} \longrightarrow \Lzero$
is given by the product
$$      |z^{-1}(c)|
    =
        \prod_{i=2}^{k}
               Z([x_{i-1},x_{i}])   .   $$
\label{theorem_Bayer_Sturmfels}
\end{theorem}

\subsection{The $\cd$-index}

Let $P$ be a graded poset of rank $n+1$
with rank function $\rho$.
For $S = \{s_{1} < \cdots < s_{k-1}\}$
a subset of
$\{1, \ldots, n\}$ define $f_{S}$ to be
the number of chains
$c = \{\hz = x_{0} < x_{1} < \cdots < x_{k} = \ho\}$
that have elements with ranks in the set $S$,
that is, 
$$    f_{S}
   =
      |\{ c \:\: : \:\:
          \rho(x_{1}) = s_{1}, \ldots, 
          \rho(x_{k-1}) = s_{k-1} \}|   .  $$
Observe that $f_{S}$ is the number
of maximal chains in the rank-selected poset $P(S)$.
The flag $h$-vector is obtained by the relation
(here we also present its inverse)
$$       h_{S} 
      =
         \sum_{T \subseteq S}
             (-1)^{|S-T|} \cdot f_{T} 
 \:\:\:\: \mbox{ and } \:\:\:\:
         f_{S} 
      =
         \sum_{T \subseteq S}
             h_{T}               .  $$
Recall that by Philip Hall's theorem
the M\"obius function of the $S$-rank-selected poset
$P(S)$ is given by
$\mu(P(S)) = (-1)^{|S|-1} \cdot h_{S}$.

Let $\av$ and $\bv$ be two non-commutative variables
each having degree $1$.
For $S$ a subset of $\{1, \ldots, n\}$ let $u_{S}$ be the monomial
$u_{S} = u_{1} \cdots u_{n}$ where
$u_{i} = \bv$ if $i \in S$
and
$u_{i} = \av$ if $i \not\in S$.
Then the {\em $\ab$-index} is the noncommutative polynomial defined by
$$  \Psi(P) = \sum_{S} h_{S} \cdot u_{S} ,  $$
where the sum is over all subsets $S \subseteq \{1, \ldots, n\}$.
Observe that the $\ab$-index of a poset $P$ of rank $n+1$
is a homogeneous polynomial of degree $n$.

A poset $P$ is {\em Eulerian} if
every interval $[x,y]$, where $x < y$,
satisfies the Euler-Poincar\'e
relation, that is, there are  the same number of elements
of odd as even rank.
Equivalently, the M\"obius function of $P$ is given by
$\mu(x,y) = (-1)^{\rho(x,y)}$
for all $x \leq y$ in $P$.
The quintessential result is that
the $\ab$-index 
of Eulerian posets has the following
form.
\begin{theorem}
The $\ab$-index of an Eulerian poset $P$
can be expressed in terms of
the noncommutative variables
$\cv = \av + \bv$
and $\dv = \av \bv + \bv \av$.
\end{theorem}
This theorem
was originally proved for face lattices of convex polytopes
by Bayer and Klapper~\cite{Bayer_Klapper}.
Stanley provided a proof for all
Eulerian posets~\cite{Stanley_d}.
There are proofs which have both used and revealed  the underlying
algebraic structure.
See for instance~\cite{Ehrenborg_k-Eulerian,Ehrenborg_Readdy_homology}.
When the $\ab$-index $\Psi(P)$ is written in terms of
$\cv$ and $\dv$,
the resulting polynomial is called the {\em $\cd$-index}.
There are linear relations 
holding among the entries of the flag $f$-vector
of an Eulerian poset, known as the generalized
Dehn-Sommerville relations; see~\cite{Bayer_Billera}.
The importance of the $\cd$-index is that it removes all of these
linear redundancies among the flag $f$-vector entries.

Observe that the variables $\cv$ and $\dv$ have degrees
$1$ and $2$, respectively.
Thus the $\cd$-index of a poset of rank $n+1$ is
a homogeneous polynomial of degree $n$ in the
noncommutative variables $\cv$ and $\dv$.
Define the reverse of an $\ab$-monomial
$u = u_{1} u_{2} \cdots u_{n}$ to be
$u^{*} = u_{n} \cdots u_{2} u_{1}$
and extend by linearity to an involution on $\zab$.
Since $\cv^{*} = \cv$ and $\dv^{*} = \dv$,
this involution applied to a $\cd$-monomial
simply reverses the $\cd$-monomial.
Finally, for a graded poset $P$
we have  $\Psi(P)^{*} = \Psi(P^{*})$.

A direct approach to describe the $\ab$-index of a
poset $P$ is to give each chain a weight and
then sum over all chains.
For a chain
$c = \{\hz = x_{0} < x_{1} < \cdots < x_{k} = \ho\}$
in the poset $P$, define its {\em weight} to be
\begin{equation}
      \wt(c)
   =
      (\av-\bv)^{\rho(x_{0},x_{1})-1}
        \cdot
      \bv
        \cdot
      (\av-\bv)^{\rho(x_{1},x_{2})-1}
        \cdot
      \bv
        \cdots
      \bv
        \cdot
      (\av-\bv)^{\rho(x_{k-1},x_{k})-1}  ,
\label{equation_weight}
\end{equation}
where $\rho(x,y)$ denotes the rank difference
$\rho(y) - \rho(x)$.
Then $\ab$-index of $P$ is given by
$$
     \Psi(P) = \sum_{c} \wt(c),
$$
where the sum is over all chains $c$ in the poset $P$.

Finally, a third description of the $\ab$-index is
Stanley's recursion for the $\ab$-index of a graded
poset~\cite[Equation~(7)]{Stanley_d}.  It is:
\begin{equation}
\label{equation_Stanley_recursion}
 \Psi(P) = (\av-\bv)^{\rho(P)-1}
         + \sum_{\hz < x < \ho}
              (\av-\bv)^{\rho(x)-1} \cdot \bv \cdot \Psi([x,\ho]).
\end{equation}
The initial condition for
 this recursion is the unique poset of rank $1$,
$B_{1}$, where $\Psi(B_{1}) = 1$.

\subsection{Coalgebraic techniques}

A coproduct $\Delta$ on a free $\Zzz$-module $C$
is a linear map $\Delta : C \longrightarrow C \tensor C$.
In order to be explicit, we use the Sweedler notation~\cite{Sweedler}
for writing the coproduct. To explain this notation, notice that
$\Delta(w)$ is an element of $C \tensor C$ and thus has the form
$$   \Delta(w) = \sum_{i=1}^{k} w_{1}^{i} \tensor w_{2}^{i}   ,  $$
where $k$ is the number of terms and $w_{1}^{i}$ and $w_{2}^{i}$
belong to $C$. Since all the maps that are applied to $\Delta(w)$
treat each term the same, the Sweedler notation drops the index $i$
and one writes
$$   \Delta(w) = \sum_{w} w_{(1)} \tensor w_{(2)}   .  $$
Informally, this sum should be thought of as all the ways of breaking
the element $w$ in two pieces, where the
first piece is denoted by  $w_{(1)}$
and the second by $w_{(2)}$.
The Sweedler notation for the expression
$(\Delta \tensor \id) \circ \Delta$,
where $\id$ denotes the identity map,
is the following
$$    ((\Delta \tensor \id) \circ \Delta)(w)
   =
      \sum_{w} \sum_{w_{(1)}}
         w_{(1,1)} \tensor w_{(1,2)} \tensor w_{(2)}  .  $$
The right-hand side should be thought of as
first breaking $w$ into the two pieces $w_{(1)}$ and $w_{(2)}$
and then breaking
$w_{(1)}$ into the two pieces $w_{(1,1)}$ and $w_{(1,2)}$.
See Joni and Rota for a more detailed explanation~\cite{Joni_Rota}.

The coproduct $\Delta$ is coassociative
if
$(\Delta \tensor \id) \circ \Delta = 
 (\id \tensor \Delta) \circ \Delta$.
The Sweedler notation expresses coassociativity as
$$   \sum_{w} \sum_{w_{(1)}}
         w_{(1,1)} \tensor w_{(1,2)} \tensor w_{(2)}
   =
     \sum_{w} \sum_{w_{(2)}}
         w_{(1)} \tensor w_{(2,1)} \tensor w_{(2,2)}  .  $$
Informally coassociativity states that 
all the possible ways to break $w$ into two pieces and
then breaking the first piece
into the two pieces
is equivalent to
all the ways to break $w$ into two pieces and
then break the second piece
into two pieces.
Compare coassociativity with associativity of
a multiplication map $m: A \tensor A \longrightarrow A$
on an algebra $A$.

Assuming coassociativity, Sweedler notation
simplifies to
$$   \Delta^{2}(w) = \sum_{w} w_{(1)} \tensor w_{(2)} \tensor w_{(3)} , $$
where $\Delta^{2}$ is defined as
$(\Delta \tensor \id) \circ \Delta = (\id \tensor \Delta) \circ \Delta$,
and the three pieces have been renamed as
$w_{(1)}$, $w_{(2)}$ and $w_{(3)}$.
Coassociativity allows one to define
the $k$-ary coproduct
$\Delta^{k-1} : C \longrightarrow C^{\tensor k}$
by the recursion
$\Delta^{0} = \id$
and
$\Delta^{k} = (\Delta^{k-1} \tensor \id) \circ \Delta$.
The Sweedler notation for the $k$-ary coproduct is
$$   \Delta^{k-1}(w)
   = 
     \sum_{w}
         w_{(1)} \tensor w_{(2)} \tensordots w_{(k)}  .  $$

Let $\zab$ denote the polynomial ring in the non-commutative
variables $\av$ and $\bv$.
We define a coproduct~$\Delta$ on the algebra $\zab$
by letting $\Delta$ satisfy the following
identities:
$\Delta(1) = 0$, $\Delta(\av) = \Delta(\bv) = 1 \tensor 1$
and the Newtonian condition 
\begin{equation}
      \Delta(u \cdot v)
    =
      \sum_{u} u_{(1)} \tensor u_{(2)} \cdot v
    +
      \sum_{v} u \cdot v_{(1)} \tensor v_{(2)}    .  
\label{equation_Newtonian}
\end{equation}
For an $\ab$-monomial $u = u_{1} u_{2} \cdots u_{n}$
we have that
$$    \Delta(u)
    =
      \sum_{i=1}^{n}
               u_{1} \cdots u_{i-1}
            \tensor
               u_{i+1} \cdots u_{n}    .    $$
The fundamental result for this coproduct is that
the $\ab$-index is a coalgebra homomorphism~\cite{Ehrenborg_Readdy_c}.
We express this result as the following identity.
\begin{theorem}[Ehrenborg--Readdy]
For a graded poset $P$
with $\ab$-index $w = \Psi(P)$ and for any $k$-multilinear map~$M$ 
on $\zab$, 
the following coproduct identity holds:
$$   \sum_{c}
           M(\Psi([x_{0},x_{1}]),
             \Psi([x_{1},x_{2}]),
                \ldots,
             \Psi([x_{k-1},x_{k}]))
  =
     \sum_{w}
           M(w_{(1)},
             w_{(2)},
                \ldots,
             w_{(k)})    ,  $$
where the first sum is over all
chains $c = \{\hz = x_{0} < x_{1} < \cdots < x_{k} = \ho\}$
of length $k$ 
and the second sum is the Sweedler notation of the $k$-ary coproduct
of $w$, that is, $\Delta^{k-1}$.
\label{theorem_Ehrenborg_Readdy}
\end{theorem}

\subsection{The $\cd$-index of the face poset of a central arrangement}

We recall the definition of the 
omega map~\cite{Billera_Ehrenborg_Readdy_om}.

\begin{definition}
The linear map $\omega$ from $\zab$ to $\zcd$
is formed by 
first replacing every occurrence of $\av\bv$ 
in a given 
$\ab$-monomial 
by $2\dv$
and then replacing the remaining letters by $\cv$.
\end{definition}
For a central hyperplane arrangement~$\HH$
the $\cd$-index of the face poset is computed
as follows~\cite{Billera_Ehrenborg_Readdy_om}.
\begin{theorem}[Billera--Ehrenborg--Readdy]
Let $\HH$ be
a central hyperplane arrangement
with intersection lattice $\LL$
and face lattice $T$. Then the $\cd$-index
of the
face lattice $T$ is given by
$$  \Psi(T) = \omega(\av \cdot \Psi(\LL))^{*}  .  $$
\label{theorem_Billera_Ehrenborg_Readdy}
\end{theorem}

We review the basic ideas behind the proof of this theorem.
We will refer back to them when we
prove similar results for toric and affine arrangements
in Sections~\ref{section_toric}
and~\ref{section_affine}.

Define three linear operators $\kappa$, $\beta$ and $\eta$ on
$\zab$ by
$$
\kappa(v) = \begin{cases}
(\av-\bv)^{m} & \text{if $v=\av^{m}$ for some $m \geq 0$,} \\
0 & \text{otherwise,}
\end{cases}
$$
$$
\beta(v) = \begin{cases}
(\av-\bv)^{m} & \text{if $v=\bv^{m}$ for some $m \geq 0$,} \\
0 & \text{otherwise,}
\end{cases}
$$
and 
$$
\eta(v) = \begin{cases}
2 \cdot (\av-\bv)^{m+k} & \text{if $v = \bv^{m} \av^{k}$ for some $m, k \geq 0$,} \\
0 & \text{otherwise.}
\end{cases}
$$
Observe that $\kappa$ and $\beta$ are both algebra maps.
The following relations hold for a poset $P$;
see~\cite[Section 5]{Billera_Ehrenborg_Readdy_om}:
\begin{eqnarray}
\kappa(\Psi(P)) & = & (\av-\bv)^{\rho(P)-1} ,                  
\label{equation_kappa} \\
\beta(\Psi(P))  & = & Z_{b}(P) \cdot (\av-\bv)^{\rho(P)-1} ,   
\label{equation_beta} \\
\eta(\Psi(P))   & = & Z(P) \cdot (\av-\bv)^{\rho(P)-1} .
\label{equation_eta}
\end{eqnarray}
For $k \geq 1$ the operator $\varphi_{k}$
is defined by the coalgebra expression
$$   \varphi_{k}(v)
   =
     \sum_{v}
         \kappa(v_{(1)}) \cdot \bv \cdot
         \eta(v_{(2)}) \cdot \bv \cdots \bv \cdot
         \eta(v_{(k)})   ,  $$
where the coproduct splits $v$ into $k$ parts.
Finally $\varphi$ is defined as the sum
$$       \varphi(v) = \sum_{k \geq 1} \varphi_{k}(v)  .   $$
Note that 
in this expression
only a finite number of terms are non-zero.
The connection with hyperplane arrangements is given by
the following proposition.
\begin{proposition}
The $\ab$-index of the lattice of faces of a central hyperplane
arrangement is given by
$$   \Psi(T) = \varphi( \Psi( \Lzero ) )^{*}   .   $$
\label{proposition_varphi}
\end{proposition}
The function $\varphi$ satisfies the functional equation
$$     \varphi(v)
     = 
       \kappa(v)
     + 
       \sum_{v} \varphi(v_{(1)}) \cdot \bv \cdot \eta(v_{(2)}) .  $$
From this relation it follows that
the function $\varphi$ satisfies the initial conditions
$\varphi(1) = 1$ and $\varphi(\bv) = 2 \cdot \bv$
and the recursions:
\begin{eqnarray}
\varphi(v \cdot \av) & = & \varphi(v) \cdot \cv , 
\label{equation_varphi_a} \\
\varphi(v \cdot \bv \bv) & = & \varphi(v \cdot \bv) \cdot \cv , 
\label{equation_varphi_b_b} \\
\varphi(v \cdot \av \bv) & = & \varphi(v) \cdot 2 \dv ,
\label{equation_varphi_a_b}
\end{eqnarray}
for an $\ab$-monomial $v$;
see~\cite[Section~5]{Billera_Ehrenborg_Readdy_om}.
These recursions culminate in the following result.
\begin{proposition}
\label{proposition_culminate}
For an $\ab$-monomial $w$ that begins with $\av$,
the two maps $\varphi$ and $\omega$ coincide,
that is,
$\varphi(w) = \omega(w)$.
\end{proposition}
Finally, Theorem~\ref{theorem_Billera_Ehrenborg_Readdy}
follows by Proposition~\ref{proposition_culminate} and 
from the fact that
$\Psi(\Lzero) = \av \cdot \Psi(\LL)$.

\subsection{Regular subdivisions of manifolds}

The face poset $P(\Omega)$ of a cell complex $\Omega$ is
the set of all cells in $\Omega$ together
with a minimal element~$\hz$ and a maximal element $\ho$.
One partially orders two
cells $\tau$ and $\sigma$
by requiring that
$\tau < \sigma$ if the cell $\tau$ is contained in $\overline{\sigma}$,
the closure of $\sigma$.
In order to define a regular cell complex,
consider the cell complex $\Omega$ embedded in 
Euclidean space $\Rrr^{n}$.
This condition is compatible with toric cell complexes
since the $n$-dimensional torus can be embedded 
in $2n$-dimensional Euclidean space.
Let $B^{n}$ denote the ball
$\{x \in \Rrr^{n} \: : \: x_{1}^{2} + \cdots + x_{n}^{2} \leq 1\}$
and
let $S^{n-1}$ denote the sphere
$\{x \in \Rrr^{n} \: : \: x_{1}^{2} + \cdots + x_{n}^{2} = 1\}$.
A cell complex $\Omega$ 
is {\em regular} if 
(i) $\Omega$ consists of a finite number of cells,
(ii) for every cell $\sigma$ of $\Omega$
the pair $(\overline{\sigma}, \overline{\sigma} - \sigma)$
is homeomorphic
to a pair $(B^{k},S^{k-1})$ for some integer $k$,
and
(iii)
the boundary
$\overline{\sigma} - \sigma$
is the disjoint union of smaller cells in $\Omega$.
See Section~3.8 in~\cite{Stanley_EC_I} for more details.
For a discussion of regular cell complexes not embedded
in $\Rrr^n$,
see~\cite{Bjorner_topological_methods}.

The face poset of a regular subdivision of the sphere
is an Eulerian face poset
and hence has a $\cd$-index.
For regular subdivisions of compact manifolds,
a similar result holds.
This was independently observed by
Ed~Swartz~\cite{Swartz}.
\begin{theorem}
Let $\Omega$ be a regular cell complex whose geometric realization
is a compact $n$-dimensional manifold $M$. 
Let $\chi(M)$ denote the Euler characteristic of $M$.
Then the $\ab$-index
of the face poset $P$ of $\Omega$ has the following form.
\begin{enumerate}
\item[(i)]
If $n$ is odd then
$P$ is an Eulerian poset and hence
$\Psi(P)$ can written in terms of $\cv$ and $\dv$.

\item[(ii)]
If $n$ is even then $\Psi(P)$ has the form
$$      \Psi(P) 
    = 
        \left( 1-\frac{\chi(M)}{2}
        \right) \cdot (\av-\bv)^{n+1}
      + 
        \frac{\chi(M)}{2} \cdot \cv^{n+1}
      + 
        \Phi   ,   $$
where $\Phi$ is a homogeneous $\cd$-polynomial of degree $n+1$
and $\Phi$ does not contain the term $\cv^{n+1}$.
\end{enumerate}
\label{theorem_manifold}
\end{theorem}
\begin{proof}
Observe that the poset $P$ has rank $n+2$.
By~\cite[Theorem~3.8.9]{Stanley_EC_I} we know that
every interval~$[x,y]$ strictly contained in $P$
is Eulerian. When the rank of $P$ is odd this implies
that $P$ is also Eulerian; see~\cite[Exercise 69c]{Stanley_EC_I}.
Hence in this case the $\ab$-index of $P$
can be expressed as a $\cd$-index.
When $n$ is even, we use~\cite[Theorem~4.2]{Ehrenborg_k-Eulerian}
to conclude that
the $\ab$-index of $P$ belongs to
$\Rrr\langle\cv,\dv,(\av-\bv)^{n+1}\rangle$.
Since
$\Psi(P)$ has degree $n+1$,
the $\ab$-index $\Psi(P)$ can be written in the form
$$      \Psi(P) 
    = 
        c_{1} \cdot (\av-\bv)^{n+1}
      + 
        c_{2} \cdot \cv^{n+1}
      + 
        \Phi   ,   $$
where $\Phi$ is a homogeneous $\cd$-polynomial of degree $n+1$
that does not contain any
$\cv^{n+1}$ terms. 
By looking at the coefficients
of $\av^{n+1}$ and $\bv^{n+1}$, we have
$c_{1} + c_{2} = 1$ and
$c_{2} - c_{1} = \mu(P) = \chi(M) - 1$,
where the last identity is again~\cite[Theorem~3.8.9]{Stanley_EC_I}.
Solving for $c_{1}$ and $c_{2}$ proves the result.
\end{proof}

\begin{corollary}
Let $P$ be the face poset of a regular cell complex whose geometric realization
is a compact $n$-dimensional manifold $M$. 
If $n$ is odd then the flag $h$-vector of $P$ is symmetric, that is,
$h_{S} = h_{\overline{S}}$.
If $n$ is even the flag $h$-vector of $P$ satisfies
$$     h_{S} - h_{\overline{S}}
    =
       (-1)^{|S|} \cdot \left( 2 - \chi(M) \right)  . $$
\end{corollary}

For the $n$-dimensional torus Theorem~\ref{theorem_manifold}
can be expressed as follows.
\begin{corollary}
Let $\Omega$ be a regular cell complex whose geometric realization
is the $n$-dimensional torus~$T^{n}$. Then the $\ab$-index
of the face poset $P$ of $\Omega$ has the following form:
$$      \Psi(P) 
    = 
        (\av-\bv)^{n+1}
      + 
        \Phi   ,   $$
where $\Phi$ is a homogeneous $\cd$-polynomial of degree $n+1$
and $\Phi$ does not contain the term $\cv^{n+1}$.
\end{corollary}
\begin{proof}
When $n$ is even this is Theorem~\ref{theorem_manifold}.
When $n$ is odd this is Theorem~\ref{theorem_manifold}
together with the two facts that
$\chi(T^{n}) = 0$ and
$(\av-\bv)^{n+1} = (\cv^{2} - 2 \dv)^{(n+1)/2}$.
\end{proof}

\section{Toric arrangements}
\label{section_toric}

\subsection{Toric subspaces and arrangements}

\begin{figure}
\setlength{\unitlength}{0.5mm}
\begin{center}
\begin{picture}(100,100)(-20,-20)

\put(0,-20){\line(0,1){100}}
\put(-20,0){\line(1,0){100}}


\put(0,60){\line(1,0){60}}
\put(60,0){\line(0,1){60}}

\put(-20,-10){\line(2,1){100}}
\put(-10,-20){\line(1,2){50}}

\thicklines

\put(0,0){\line(2,1){60}}
\put(0,0){\line(1,2){30}}
\put(0,30){\line(2,1){60}}
\put(30,0){\line(1,2){30}}

\put(28,28){{\small 1}}

\put(53,35){{\small 2}}
\put(9,46){{\small 2}}
\put(21,4){{\small 2}}

\put(4,21){{\small 3}}
\put(46,9){{\small 3}}
\put(35,53){{\small 3}}

\end{picture}
\hspace*{15 mm}
\begin{picture}(100,100)(-5,-5)

\put(45,0){\circle*{3}}
\multiput(30,30)(30,0){2}{\put(0,0){\circle*{3}}}
\multiput(15,60)(30,0){3}{\put(0,0){\circle*{3}}}
\put(45,90){\circle*{3}}

\put(45,0){\line(-1,2){15}}
\put(45,0){\line(1,2){15}}

\put(30,30){\line(-1,2){15}}
\put(30,30){\line(1,2){15}}
\put(30,30){\line(3,2){45}}

\put(60,30){\line(-3,2){45}}
\put(60,30){\line(-1,2){15}}
\put(60,30){\line(1,2){15}}

\put(45,90){\line(-1,-1){30}}
\put(45,90){\line(0,-1){30}}
\put(45,90){\line(1,-1){30}}
\end{picture}
\end{center}
\caption{A toric line arrangement
which subdivides
the torus $T^{2}$ into a non-regular cell complex
and its intersection poset.}
\label{figure_toric_one}
\end{figure}
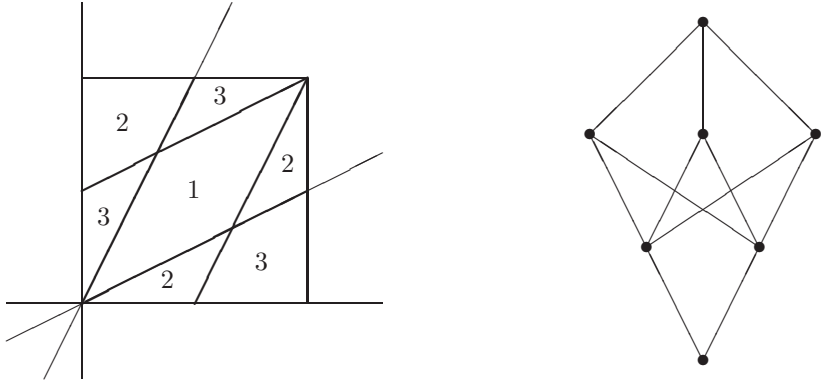

The $n$-dimensional torus $T^{n}$ is defined as the quotient
$\Rrr^{n}/\Zzz^{n}$.
Recall that the torus $T^{n}$ is an abelian group.
When identifying the torus $T^{n}$ with the set $[0,1)^{n}$,
the group structure is componentwise addition modulo $1$.
\begin{lemma}
Let $V$ be a $k$-dimensional affine subspace
in $\Rrr^{n}$ with rational coefficients. That is, $V$ has the form
$$     V 
   =
       \{ \vec{v} \in \Rrr^{n} \:\: : \:\: A \vec{v} = \vec{b}\} , $$
where the matrix $A$ has rational entries and the vector $\vec{b}$
is allowed to have real entries.
Then the image of $V$ under the
quotient map $\Rrr^{n} \to \Rrr^{n}/\Zzz^{n}$,
denoted by $\overline{V}$,
is a $k$-dimensional torus.
\end{lemma}
\begin{proof}
By translating $V$, we may assume that
the vector $\vec{b}$ is the zero vector,
and therefore
$V$ is a subspace.
In this case, the intersection of $V$ with the integer lattice $\Zzz^{n}$
is a subgroup of the free abelian group $\Zzz^{n}$.
Since the matrix $A$ has all rational entries,
the rank of this subgroup is $k$,
that is, the subgroup is isomorphic to $\Zzz^{k}$.
Hence the image $\overline{V}$
is the quotient $V/(V \cap \Zzz^{n})$,
which is isomorphic to
the quotient $\Rrr^{k}/\Zzz^{k}$, that is,
a $k$-dimensional torus.
\end{proof}

We call the image $\overline{V}$ 
a {\em toric subspace} of the torus $T^{n}$
because it is homeomorphic to some $k$-dimensional torus.
When we remove the condition that the matrix $A$ is rational, 
the image is not necessarily homeomorphic to a torus.

The intersection of two toric subspaces is in general not a toric subspace,
but instead is the disjoint union of a finite number of
toric subspaces.
For two affine subspaces $V$ and $W$ with rational coefficients,
we have that $\overline{V \cap W} \subseteq \overline{V} \cap \overline{W}$.
In general, this containment is strict.

Define the translate of a toric subspace $U$ by a point $x$
on the torus to be the toric subspace $U + x = \{ u+x  : u \in U \}$.
Alternatively, one may lift the toric subspace to an affine subspace
in Euclidean space, translate it and then map back to the torus.
Then for two toric subspaces $V$ and $W$, their intersection
has the form
$$     V \cap W = \bigcup_{p=1}^{r} (U + x_{p})  , $$
where $U$ is a toric subspace, $r$ is a non-negative integer
and $x_{1}, \ldots, x_{r}$ are points on the torus $T^{n}$.

\begin{figure}
\setlength{\unitlength}{0.5mm}
\begin{center}
\begin{picture}(100,100)(-20,-20)

\put(0,-20){\line(0,1){100}}
\put(-20,0){\line(1,0){100}}

\put(0,60){\line(1,0){60}}
\put(60,0){\line(0,1){60}}

\put(-20,-10){\line(2,1){100}}
\put(-6,-18){\line(1,3){32}}
\put(-20,12){\line(1,0){100}}

\thicklines

\put(0,0){\line(2,1){60}}
\put(0,30){\line(2,1){60}}
\put(0,0){\line(1,3){20}}
\put(20,0){\line(1,3){20}}
\put(40,0){\line(1,3){20}}
\put(0,12){\line(1,0){60}}
\end{picture}
\hspace*{15 mm}
\begin{picture}(100,100)(-5,-5)

\put(45,0){\circle*{3}}
\multiput(15,30)(30,0){3}{\put(0,0){\circle*{3}}}
\multiput(0,60)(15,0){7}{\put(0,0){\circle*{3}}}
\put(45,90){\circle*{3}}

\put(45,0){\line(-1,1){30}}
\put(45,0){\line(0,1){30}}
\put(45,0){\line(1,1){30}}

\put(15,30){\line(-1,2){15}}
\put(15,30){\line(0,1){30}}
\put(15,30){\line(1,2){15}}
\put(15,30){\line(1,1){30}}
\put(15,30){\line(3,2){45}}

\put(45,30){\line(-3,2){45}}
\put(45,30){\line(-1,1){30}}
\put(45,30){\line(-1,2){15}}
\put(45,30){\line(0,1){30}}
\put(45,30){\line(1,2){15}}
\put(45,30){\line(1,1){30}}
\put(45,30){\line(3,2){45}}

\put(75,30){\line(-1,2){15}}
\put(75,30){\line(0,1){30}}
\put(75,30){\line(1,2){15}}

\put(45,90){\line(-3,-2){45}}
\put(45,90){\line(-1,-1){30}}
\put(45,90){\line(-1,-2){15}}
\put(45,90){\line(0,-1){30}}
\put(45,90){\line(1,-2){15}}
\put(45,90){\line(1,-1){30}}
\put(45,90){\line(3,-2){45}}
\end{picture}
\end{center}
\caption{A toric line arrangement
and its intersection poset.}
\label{figure_toric_two}
\end{figure}
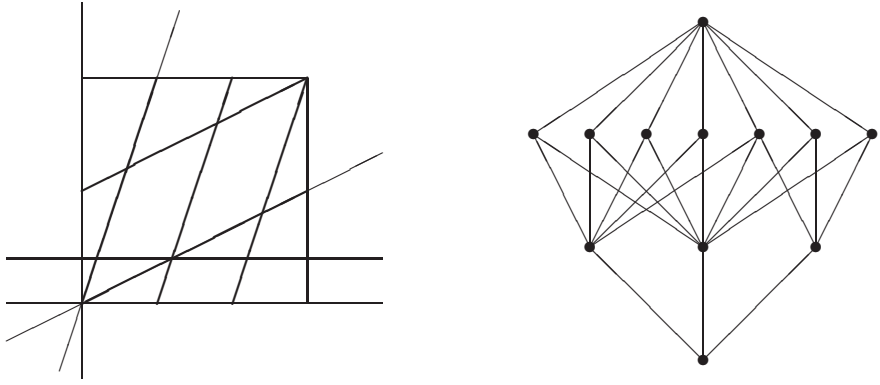

A {\em toric hyperplane arrangement} $\HH = \{H_{1}, \ldots, H_{m}\}$
is a finite collection of toric hyperplanes.
Define the {\em intersection poset} $\PP$
of a toric arrangement
to be the set of all connected components
arising from
all possible intersections of the toric hyperplanes,
that is, 
all connected components of
$\bigcap_{i \in S} H_{i}$
where $S \subseteq \{1, \ldots, m\}$,
together with the empty set.
Order the elements of
the intersection poset $\PP$ by reverse inclusion,
that is, the torus $T^{n}$ is the minimal element of $\PP$
corresponding to the empty intersection,
and the empty set is the maximal element.
A toric subspace $V$ is contained in the intersection
poset $\PP$ if there are
toric hyperplanes $H_{i_{1}}, \ldots, H_{i_{k}}$
in the arrangement such that
$V \subseteq H_{i_{1}} \cap \cdots \cap H_{i_{k}}$
and there is no toric subspace $W$
satisfying
$V \subset W \subseteq H_{i_{1}} \cap \cdots \cap H_{i_{k}}$.
In other words,
$V$ has to be a maximal toric subspace in some intersection
of toric hyperplanes from the arrangement.

The notion of using the intersection poset can be found
in
work of Zaslavsky, where he considers
topological dissections~\cite{Zaslavsky_paper}.
In this setting there is not an intersection
lattice, but rather an intersection poset.

To every toric hyperplane arrangement $\HH = \{H_{1}, \ldots, H_{m}\}$
there is an associated periodic hyperplane arrangement $\widetilde{\HH}$
in the Euclidean space $\Rrr^{n}$. Namely, the inverse
image of the toric hyperplane~$H_{i}$ under the quotient map
$\Rrr^{n} \to \Rrr^{n}/\Zzz^{n}$ is the union of parallel integer
translates of a real hyperplane.
Let~$\widetilde{\HH}$
be the collection of all these integer translates.
Observe that every face of the toric arrangement~$\HH$
can be lifted to a parallel class of faces in
the periodic real arrangement $\widetilde{\HH}$.

For a toric hyperplane arrangement $\HH$
define the {\em toric characteristic polynomial}
to be
$$   \chi(\HH; t)
   =
     \sum_{\onethingatopanother{x \in \PP}{x \neq \emptyset}}
             \mu(\hz,x) \cdot t^{\dim(x)}  .   $$
Also for a toric hyperplane arrangement $\HH$ define
$T_{t}$ to be the face poset of the induced subdivision
of the torus $T^{n}$. Note that $T_{t}$ is a graded poset
of rank $n+2$: the minimal element $\hz$ is the empty face,
the maximal element $\ho$ is the torus and the rank of
face $x$ is given by $\dim(x) + 1$.

\begin{example}
{\rm
Consider the line arrangement consisting
of the two lines $y = 2 \cdot x$
and $x = 2 \cdot y$ in the plane $\Rrr^{2}$.
In $\Rrr^{2}$ they intersect in one point,
namely
the origin,
whereas on the torus $T^{2}$ they intersect
in three points,
namely
$(0,0)$, $(2/3,1/3)$ and $(1/3,2/3)$.
The characteristic polynomial
is given by $\chi(\HH;t) = t^{2} - 2 \cdot t + 3$.
However, this arrangement is not
regular, since the induced subdivision of $T^{2}$
is not regular. The boundary of each region
is a wedge of two circles.
See Figure~\ref{figure_toric_one}.
}
\label{example_toric_one}
\end{example}

\begin{example}
{\rm
Consider the line arrangement consisting
of the three lines $y = 3 \cdot x$,
$x = 2 \cdot y$ and
$y = 1/5$.
It subdivides the torus into a regular cell complex.
The subdivision and the associated
intersection poset are shown in 
Figure~\ref{figure_toric_two}.
The characteristic polynomial
is given by $\chi(\HH;t) = t^{2} - 3 \cdot t + 8$.
Furthermore, the $\ab$-index of the
subdivision of the torus is given by
$\Psi(T_{t}) = (\av-\bv)^{3} + 7 \cdot \mdc + 8 \cdot \mcd$,
as the following calculation shows.
\newcommand{\ts}{\:\:\:\:\:\:}
$$
  \begin{array}{c r r c r c c}
 S & f_{S} & h_{S} & u_{S} & (\av-\bv)^{3} 
   & 7 \cdot \mdc & 8 \cdot \mcd \\ \hline
\emptyset &  1 &  1 & \av\av\av &  1 \ts & 0 & 0 \\
\{1\}     &  7 &  6 & \bv\av\av & -1 \ts & 7 & 0 \\
\{2\}     & 15 & 14 & \av\bv\av & -1 \ts & 7 & 8 \\
\{3\}     &  8 &  7 & \av\av\bv & -1 \ts & 0 & 8 \\
\{1,2\}   & 30 &  9 & \bv\bv\av &  1 \ts & 0 & 8 \\
\{1,3\}   & 30 & 16 & \bv\av\bv &  1 \ts & 7 & 8 \\
\{2,3\}   & 30 &  8 & \av\bv\bv &  1 \ts & 7 & 0 \\
\{1,2,3\} & 60 & -1 & \bv\bv\bv & -1 \ts & 0 & 0 \\
  \end{array} $$
Recall that
$\mdc = \maba + \nabb + \nbaa + \mbab$
and
$\mcd = \naab + \maba + \mbab + \nbba$.
Here in the last three columns we indicate
the contribution of a given term to each $\ab$-monomial.
Observe that the sum of the last three columns gives
the flag $h$-vector entries.
}
\label{example_toric_two}
\end{example}

We now give a natural interpretation of the toric characteristic polynomial.
Recall that the intersection of toric subspaces
is the disjoint union of toric subspaces that are translates of each other.
Let~$G$ be the collection of finite intersections of toric
subspaces of the $n$-dimensional torus $T^{n}$,
that is, $G$ consists of sets of the form
$V = W_{1} \capdots W_{q}$, where $W_{1}, \ldots, W_{q}$ are
toric subspaces.
Such a set~$V$ can be written as a union,
more precisely,
$V = \bigcup_{p=1}^{r} (U + x_{p})$,
where $U$ is a toric subspace, $r$ a non-negative integer,
and $x_{1}, \ldots, x_{r}$ are points on the torus.
Observe that the empty set $\emptyset$
and the torus~$T^{n}$ belong to $G$.
Furthermore, $G$ is closed under finite intersections.
Let $L$ be the distributive lattice consisting of all
subsets of the torus $T^{n}$ that are obtained from
the collection~$G$ by finite intersections, finite unions
and complements. The set $G$ is the generating set for the lattice~$L$.
A {\em valuation}~$v$ is a function on the lattice $L$
to an abelian group
satisfying
$v(\emptyset) = 0$ and $v(A) + v(B) = v(A \cap B) + v(A \cup B)$
for all sets $A, B \in L$.

The next theorem is analogous to Theorem~2.1
in~\cite{Ehrenborg_Readdy_valuation_1}.
The proof here is more involved
due to the fact that the collection of
toric subspaces is not closed under intersections.
\begin{theorem}
There is a valuation $v$ on the distributive lattice~$L$
to $\Zzz[t]$
such that 
for a $k$-dimensional toric subspace $V$
its valuation is $v(V) = t^{k}$.
\end{theorem}
\begin{proof}
Define the function $v$ on the generating set $G$ by
$$     v\left(\bigcup_{p=1}^{r} (U + x_{p})\right)
    =
       r \cdot t^{k}   ,  $$
where we assume that $U$ is a $k$-dimensional toric subspace
and the $r$ translates
$U + x_{1}, \ldots, U + x_{r}$ are pairwise disjoint.
Observe that the function $v$ is additive with respect to
disjoint unions, that is, for elements
$V_{1}, \ldots, V_{m}$ in $G$
which are pairwise disjoint
and 
$V_{1} \cupdots V_{m} \in G$.
In this case, each $V_{i}$ is a disjoint union of translates of
the same affine subspace $U$ and both sides of the identity
$v(V_{1}) + \cdots + v(V_{m})
    =
 v(V_{1} \cupdots V_{m})$
count the number of translates of $U$ 
times $t^{\dim(U)}$.

Groemer's integral theorem~\cite{Groemer}
(see also~\cite[Theorem~2.2.1]{Klain_Rota})
states that a function $v$
defined on a generating set $G$ extends to a valuation
on the distributive lattice generated by $G$ if
for all $V_{1}, \ldots, V_{m}$ in $G$ such that
$V_{1} \cupdots V_{m} \in G$, the inclusion-exclusion
formula holds:
\begin{equation}
v(V_{1} \cupdots V_{m})
  =
\sum_{i} v(V_{i})
  -
\sum_{i<j} v(V_{i} \cap V_{j})
  +
\cdots .
\label{equation_inclusion_exclusion}
\end{equation}
To verify this for our generating set $G$,
first consider the case when 
the union $V_{1} \cupdots V_{m}$
is a toric subspace.
This case implies that
$V_{1} \cupdots V_{m} = V_{i}$ for some index $i$.
It then follows that the inclusion-exclusion
formula~(\ref{equation_inclusion_exclusion}) holds
trivially.

Before considering the general case,
we introduce some notation.
For $S$ a non-empty subset of the index set $\{1, \ldots, m\}$,
let $V_{S} = \bigcap_{i \in S} V_{i}$.
Equation~(\ref{equation_inclusion_exclusion})
can then be written as
$$
v(V_{1} \cupdots V_{m})
  =
\sum_{S} (-1)^{|S|-1} \cdot v(V_{S})  ,
$$
where the sum ranges over non-empty subsets $S$ of $\{1, \ldots, m\}$.
Now assume that 
$V_{1} \cupdots V_{m}$ is the disjoint union
$(U + x_{1}) \cupdots (U + x_{r})$.
Let $V_{S,p}$ denote the intersection
$V_{S} \cap (U + x_{p})$.
Observe that
$U + x_{p} = \bigcup_{i=1}^{m} V_{\{i\},p}$
and since $U + x_{p}$ is itself a toric subspace,
we have already proved that
the inclusion-exclusion formula~(\ref{equation_inclusion_exclusion})
holds for this union.
Hence we have
\begin{eqnarray*}
v(V_{1} \cupdots V_{m})
  & = &
\sum_{p=1}^{r} v(U + x_{p}) \\
  & = &
\sum_{p=1}^{r} \sum_{S} (-1)^{|S|-1} \cdot v(V_{S,p}) \\
  & = &
\sum_{S} (-1)^{|S|-1} \cdot \sum_{p=1}^{r} v(V_{S,p}) \\
  & = &
\sum_{S} (-1)^{|S|-1} \cdot v(V_{S}) ,
\end{eqnarray*}
where $S$ ranges over all non-empty subsets of $\{1, \ldots, m\}$.
The last step follows since the
union $V_{S} = \bigcup_{p=1}^{r} V_{S,p}$
is pairwise disjoint.
\end{proof}

By M\"obius inversion we directly have the following theorem.
The proof is standard.  See the 
references~\cite{Athanasiadis,Chen,Ehrenborg_Readdy_valuation_1,Jozefiak_Sagan}.
\begin{theorem}
The characteristic polynomial of a toric arrangement is given by
$$  \chi(\HH) = v\left( T^{n} - \bigcup_{i=1}^{m} H_{i} \right) .  $$
\label{theorem_characteristic}
\end{theorem}

When each region is an open ball we can now determine the number of
regions in a toric arrangement.  The proof is analogous to the proofs
in~\cite{Ehrenborg_Readdy_valuation_1,Ehrenborg_Readdy_valuation_2}.

\begin{theorem}
Let $\HH$ be a toric hyperplane arrangement on the $n$-dimensional
torus $T^{n}$ that subdivides the torus into 
regions that are open $n$-dimensional balls.
Then the number of regions in the complement of the arrangement 
is given by $(-1)^{n} \cdot \chi(\HH; t=0)$.
\label{theorem_toric_Zaslavsky_version_1}
\end{theorem}
\begin{proof}
Observe that the Euler valuation $\varepsilon$
of a $k$-dimensional torus
is given by the Kronecker delta~$\delta_{k,0}$.
Hence the Euler valuation of a toric subspace $V$
of the $n$-dimensional torus $T^{n}$ is given
by setting $t=0$ in the valuation $v(V)$, that
is, $\varepsilon(V) = v(V)|_{t = 0}$.
Since the two valuations $\varepsilon$ and $v|_{t=0}$
are additive with respect to disjoint unions,
they agree for any member of the generating set $G$.
Hence they also agree for any member in the distributive lattice $L$.
In particular, we have that
\begin{equation}
    \varepsilon\left( T^{n} - \bigcup_{i=1}^{m} H_{i} \right) 
  =
    \left. v\left( T^{n} - \bigcup_{i=1}^{m} H_{i} \right)\right|_{t = 0} .
\label{equation_Euler_and_v}
\end{equation}
Since the Euler valuation of an open ball
is $(-1)^{n}$ and
$T^{n} - \bigcup_{i=1}^{m} H_{i}$
is a disjoint union of open balls,
the left-hand side 
of~(\ref{equation_Euler_and_v})
is $(-1)^{n}$ times the number of regions.
The right-hand side is $\chi(\HH; t=0)$ by 
Theorem~\ref{theorem_characteristic}.
\end{proof}

\addtocounter{theorem}{-5}
\begin{continuation}
{\rm
Setting $t=0$ in the characteristic polynomial
in Example~\ref{example_toric_one}
we obtain $3$, which is indeed
is the number of regions of this arrangement.
}
\end{continuation}
\addtocounter{theorem}{4}

We call a toric hyperplane arrangement $\HH = \{H_{1}, \ldots, H_{m}\}$
{\em rational}
if each hyperplane $H_{i}$
is of the form
$\vec{a}_{i} \cdot \vec{x} = b_{i}$
where the vector $\vec{a}_{i}$ has integer entries and
$b_{i}$ is an integer for $1 \leq i \leq m$.
This is equivalent to assuming every constant $b_{i}$ is rational
since every vector $\vec{a}_{i}$ was already assumed to be rational.
In what follows it will be convenient to
assume every coefficient is integral in a given rational
arrangement.

Define
$\lcm(\HH)$ to be
the least common multiple of all the
$n \times n$ minors of the $n \times m$ matrix
$(\vec{a}_{1}, \ldots, \vec{a}_{m})$.
We can now give a different interpretation
of the toric chromatic polynomial
by counting lattice points.
\begin{theorem}
For a rational hyperplane arrangement $\HH$
there exists a constant $k$ such that
for every $q > k$ where $q$ is a multiple of $\lcm(\HH)$,
the toric characteristic polynomial
evaluated at $q$ is given by
the number of lattice points in
$\left( \frac{1}{q} \Zzz \right)^{n}/\Zzz^{n}$
that do not lie on any of the toric hyperplanes $H_{i}$, that is,
$$
    \chi(\HH; t=q)
  =
    \left|
      \left( \frac{1}{q} \Zzz \right)^{n}/\Zzz^{n}
        -
      \bigcup_{i=1}^{m} H_{i}
    \right|    . 
$$
\label{theorem_lattice_points}
\end{theorem}
The condition that $q$ is a multiple of $\lcm(\HH)$
implies that every subspace $x$
in the intersection poset~$\PP$
intersects the toric lattice
$\left(\frac{1}{q} \Zzz \right)^{n}/\Zzz^{n}$
in exactly $q^{\dim(x)}$ points.
Theorem~\ref{theorem_lattice_points}
now follows by M\"obius inversion.
This theorem
is the toric analogue of
Athanasiadis'
finite field method.
See especially \cite[Theorem~2.1]{Athanasiadis_II}.

In the case when $\lcm(\HH) = 1$, the toric arrangement 
$\HH$ is
called {\em unimodular}. 
Novik, Postnikov and Sturmfels~\cite{Novik_Postnikov_Sturmfels}
state Theorem~\ref{theorem_toric_Zaslavsky_version_1}
in the special case of unimodular arrangements.
Their first proof is based upon Zaslavsky's result on the number
of bounded regions in an affine arrangement.
The second proof, due to Vic Reiner, is equivalent to our
proof for arbitrary toric arrangements.
See also the paper~\cite{Zaslavsky_paper}
by Zaslavsky, where more general arrangement
are considered.

\subsection{Graphical arrangements}

We digress in this subsection to discuss an application
to graphical arrangements.
For a graph $G$ on the vertex set $\{1, \ldots, n\}$
define the graphical arrangement $\HH_{G}$ to
be the collection of hyperplanes of the form $x_{i} = x_{j}$
for each edge $ij$ in the graph $G$.
\begin{corollary}
For a connected graph $G$ on $n$ vertices the regions
in the complement
of the graphical arrangement $\HH_{G}$ on the torus $T^{n}$
are each homotopy equivalent to the $1$-dimensional torus $T^{1}$.
Furthermore, the number of regions is given by
$(-1)^{n-1}$ times the linear coefficient of the chromatic
polynomial of $G$.
\end{corollary}
\begin{proof}
Recall the chromatic polynomial of the graph $G$ is equal to
the characteristic polynomial of the graphical arrangement
$\HH_{G}$.
The intersection lattice
of the real arrangement $\HH_{G}$ is the same
as the intersection poset of the toric arrangement $\HH_{G}$.
Translating the graphical arrangement in the
direction $(1, \ldots, 1)$ leaves the
arrangement on the torus invariant.
Since $G$ is connected this is the only direction
that leaves the arrangement invariant.
Hence each region is
homotopy equivalent to $T^{1}$.
By adding the hyperplane $x_{1} = 0$ to the arrangement
we obtain a new arrangement $\HH^{\prime}$
with the same number of regions, 
but with each region homeomorphic to a ball.
Since the intersection lattice of $\HH^{\prime}$
is just the Cartesian product of the two-element poset
with the  intersection lattice of $\HH_{G}$, we have
$$   \chi(\HH^{\prime}, t)
   = 
     (t-1) \cdot \chi(\HH_{G}, t)/t .  $$
The number of regions is obtained by setting $t=0$ in this equality.
\end{proof}

A similar statement holds for graphs that are disconnected.
The result follows from the fact that
the complement of the graphical arrangement
is the product of the complements
of each connected component.
\begin{corollary}
For a graph $G$ on $n$ vertices 
consisting of $k$ components,
the regions
in the complement
of the graphical arrangement $\HH_{G}$ on the torus $T^{n}$
are each homotopy equivalent to the $k$-dimensional torus $T^{k}$.
The number of regions is given by
$(-1)^{n-k}$ times the coefficient of $t^{k}$ in the chromatic
polynomial of $G$.
\end{corollary}

Stanley~\cite{Stanley_acyclic} proved the celebrated result
that the chromatic polynomial of a graph
evaluated at $t=-1$
is $(-1)^{n}$ times the number of acyclic orientations of the
graph.
A similar interpretation for the linear coefficient
of the chromatic polynomial is due to
Greene and Zaslavsky~\cite{Greene_Zaslavsky}:
\begin{theorem}[Greene--Zaslavsky]
Let $G$ be a connected graph and $v$ a given vertex of the graph.
The linear coefficient of the chromatic polynomial
is $(-1)^{n-1}$ times the number of acyclic orientations
of the graph such that the only sink is the vertex $v$.
\label{theorem_Greene_Zaslavsky}
\end{theorem}
\begin{proof}
It is enough to give a bijection between
regions in the complement of the graphical arrangement
on the torus $T^{n}$
and acyclic orientations with the vertex $v$ as the unique sink.
For a region $R$ of the arrangement intersect it with
the hyperplane $x_{v} = 0$ to obtain the face $S$.
Let $\HH^{\prime}$ be the arrangement~$\HH_{G}$
together with the hyperplane $x_{v} = 0$.
Lift $S$ to a face $\widetilde{S}$
in the periodic arrangement~$\widetilde{\HH^{\prime}}$
in~$\Rrr^{n}$.
Observe that $\widetilde{S}$ is the interior of a polytope.
When minimizing the linear functional
$L(x) = x_{1} + \cdots + x_{n}$
on the closure of the face $\widetilde{S}$, the optimum
is a lattice point $k = (k_{1}, \ldots, k_{n})$.
Pick a point $x = (x_{1}, \ldots, x_{n})$ 
in $\widetilde{S}$ close to the optimum,
that is, each coordinate $x_{i}$ lies in the interval
$[k_{i},k_{i}+\epsilon)$ for some small $\epsilon > 0$.

Let $y = (y_{1}, \ldots, y_{n})$ be the image of the point
$x$ on the torus $T^{n}$, that is,
$y_{i} = x_{i} \bmod 1$. Note that
each entry $y_{i}$ lies in the half open interval $[0,1)$
and that $y_{v} = 0$.
Construct an orientation of the graph $G$
by letting the edge $ij$ be oriented $i \rightarrow j$ if $y_{i} > y_{j}$.
Note that this orientation is acyclic and has the vertex $v$
as a sink. 

To show that the vertex $v$ is the unique sink,
assume that the vertex $i$ is also a sink, where $i \neq v$.
In other words, for all neighbors $j$ of the vertex $i$ we have that
$y_{i} < y_{j}$. We can continuously move the point
$x$ in $\widetilde{S}$
by decreasing the value of the $i$th coordinate $x_{i}$.
Observe that there is no hyperplane in the periodic
arrangement blocking the coordinate $x_{i}$ from passing through
the integer value $k_{i}$ and continuing down to $k_{i}-1+\epsilon$.
This contradicts the fact
that we chose the original point~$x$ close to the optimum
of the linear functional $L$. Hence the vertex $i$ cannot be a sink.

It is straightforward to verify that
this map from regions to the set of acyclic orientations
with the unique sink at $v$ is a bijection.
\end{proof}

The technique of assigning a point to every region
of a toric arrangement using a linear functional
was used by
Novik, Postnikov and Sturmfels
in their paper~\cite{Novik_Postnikov_Sturmfels}.
See their first proof of the number of regions
of a toric arrangement.

\subsection{The toric Bayer--Sturmfels result}

Define the {\em toric Zaslavsky invariant} of a graded poset $P$
having $\hz$ and $\ho$ by
$$    \Zt(P) 
   = 
      \sum_{x \mbox{ {\rm \tiny coatom of} } P}
      (-1)^{\rho(\hz,x)} \cdot \mu(\hz,x)    
   = 
      (-1)^{\rho(P) - 1} \cdot
      \sum_{x \mbox{ {\rm \tiny coatom of} } P}
      \mu(\hz,x)    .  $$

We reformulate 
Theorem~\ref{theorem_toric_Zaslavsky_version_1}
as follows.
\begin{theorem}
For a toric hyperplane arrangement $\HH$ 
on the torus $T^n$
that subdivides the torus into 
open $n$-dimensional balls,
the number of
regions is given by $\Zt(\PP)$, where $\PP$
is the intersection poset of the arrangement~$\HH$.
\label{theorem_toric_Zaslavsky_version_2}
\end{theorem}

As a corollary of Theorem~\ref{theorem_toric_Zaslavsky_version_2},
we can describe the $f$-vector of the subdivision $T_{t}$ of the torus.
For similar results for more general manifolds
see~\cite[Section~3]{Zaslavsky_paper}.
\begin{corollary}
The number of $i$-dimensional regions in the subdivision $T_{t}$
of the $n$-dimensional torus is given by the sum
$$    f_{i+1}(T_{t})
    =
      (-1)^{i}
    \cdot
      \sum_{\onethingatopanother{x \leq y}
            {\onethingatopanother{\dim(x) = i}{\dim(y) = 0}}}
              \mu(x,y)  ,
$$
where $\mu(x,y)$ denotes the M\"obius function
of the interval $[x,y]$
in the intersection poset $\PP$.
\label{corollary_f_vector_I}
\end{corollary}
\begin{proof}
Each $i$-dimensional region is contained in a unique 
$i$-dimensional subspace $x$. By restricting the arrangement
to the subspace $x$ and applying
Theorem~\ref{theorem_toric_Zaslavsky_version_1}, we have that
the number of $i$-dimensional regions in $x$ is given by
$(-1)^{i} \cdot
 \sum_{x \leq y, \dim(y) = 0}   \mu(x,y)$.
Summing over all $x$, the result follows.
\end{proof}

For the remainder of this section we will assume that
the induced subdivision of the torus is a regular cell complex.
Let $T_{t}$ be the face poset of 
the subdivision of the torus induced by the
toric arrangement.
Define the map
$\zt : T_{t}^{*} \longrightarrow \Pzero$
by sending each face to the
smallest toric subspace in the intersection poset that contains
the face
and
sending
the minimal element
in $T_{t}^{*}$ to $\hz$.
Observe
that the map $\zt$ is order- and rank-preserving,
as well as being surjective.
As in the central hyperplane arrangement case,
we view the map $\zt$ as a map from the set of chains
of $T_{t}^{*}$ to the set of chains of $\Pzero$.

Let $x$ be an element in the intersection poset $\PP$
of a toric hyperplane arrangement $\HH$.
Then the interval $[x,\ho]$ is the intersection poset
of a toric arrangement in the toric subspace $x$.
The atoms of the interval $[x,\ho]$ are the toric hyperplanes
in this smaller toric arrangement.

More interesting is the geometric interpretation
of the interval $[\hz,x]$. It is the intersection lattice
of a central hyperplane arrangement
in $\Rrr^{n - \dim(x)}$.
Without loss of generality we may assume that
$x$ contains the zero point $(0, \ldots, 0)$,
that is, when we lift the toric subspace $x$ to an affine subspace $V$
in~$\Rrr^{n}$
we may assume that $V$ is a subspace of $\Rrr^{n}$.
Any toric subspace $y$ in the interval $[\hz,x]$,
that is, a toric subspace containing $x$,
can be lifted to a subspace $W$ containing the subspace $V$.
In particular, the toric hyperplanes in $[\hz,x]$
lift to hyperplanes in $\Rrr^{n}$ containing $V$.
This lifting is a poset isomorphism and
we obtain an essential central arrangement
of dimension $n - \dim(x)$
by quotienting out by the subspace~$V$.
We conclude by noticing that an interval $[x,y]$ in
$\PP$, where $y < \ho$, is the intersection lattice
of a central hyperplane arrangement.

The toric analogue of Theorem~\ref{theorem_Bayer_Sturmfels}
is as follows.
\begin{theorem}
Let $\PP$ be the intersection poset of a toric
hyperplane arrangement
whose induced subdivision is regular.
Let $c = \{\hz = x_{0} < x_{1} < \cdots < x_{k} = \ho\}$
be a chain in $\Pzero$
with $k \geq 2$.
Then the cardinality
of the inverse image of the chain $c$ is given by the product
$$      |\zt^{-1}(c)|
    =
        \prod_{i=2}^{k-1}
               Z([x_{i-1},x_{i}])  
    \cdot
               \Zt([x_{k-1},x_{k}])     .  $$
\label{theorem_Bayer_Sturmfels_toric}
\end{theorem}
\begin{proof}
We need to count the number of ways we can select a chain
$d = \{\hz = y_0 < y_1 < \dots < y_k = \ho\}$ in $T_{t}^{*}$
such that $\zt(y_{i}) = x_{i}$.
The number of ways to select
the element $y_{k-1}$ in $T_{t}^{*}$
is the number of regions
in the arrangement restricted to the toric subspace $x_{k-1}$.
By Theorem~\ref{theorem_toric_Zaslavsky_version_2}
this can be done in $\Zt([x_{k-1},x_{k}])$ number of ways.
Observe now that
all other elements in the chain $d$
contain the face~$y_{k-1}$.

To count the number of ways
to select the element $y_{k-2}$, we follow the original argument
of Bayer--Sturmfels. 
We would like to pick the face $y_{k-2}$
such that it contains the face $y_{k-1}$ and 
it is a region in the toric subspace $x_{k-2}$. 
This is equal to the number of regions in
the central arrangement having the intersection lattice $[x_{k-2},x_{k-1}]$,
which is given by $Z([x_{k-2},x_{k-1}])$.
By iterating this procedure
until we reach the element $y_{1}$, the result follows.
\end{proof}

\begin{corollary}
The flag $f$-vector entry
$f_{S}(T_{t})$
of the face poset
$T_{t}$ of a toric arrangement
whose induced subdivision is regular subdivision of $T^{n}$
is divisible by $2^{|S|-1}$
for $S \subseteq \{1, \ldots, n+1\}$
with
$S \neq \emptyset$.
\label{corollary_evenness}
\end{corollary}
\begin{proof}
The proof follows from the fact
that the Zaslavsky invariant $Z$ is an even integer
and that a given flag $f$-vector entry is the appropriate
sum of products appearing in
Theorem~\ref{theorem_Bayer_Sturmfels_toric}.
\end{proof}

\subsection{The connection between posets and coalgebras}

For an $\ab$-monomial $v$ 
define the linear map $\lambda_{t}$ 
by letting
$$
    \lambda_{t}(v)
  =
     \begin{cases}
        (\av-\bv)^{m}   & \text{if $v=\bv^{m}$    for some $m \geq 0$,} \\
        (\av-\bv)^{m+1} & \text{if $v=\bv^{m}\av$ for some $m \geq 0$,} \\
        0               & \text{otherwise.}
\end{cases}
$$

Define the linear operator $H^{\prime}$ on $\zab$
to be the one which removes 
the last letter in each $\ab$-monomial, that is,
$H^{\prime}(w \cdot \av) = H^{\prime}(w \cdot \bv) = w$
and $H^{\prime}(1) = 0$.
We use the prime in the notation to distinguish it from the~$H$ map
defined in~\cite[Section~8]{Billera_Ehrenborg_Readdy_om}
which instead removes the first letter in each $\ab$-monomial.
From~\cite{Billera_Ehrenborg_Readdy_om} we have the
following lemma.
\begin{lemma}
For a graded poset $P$ with $\ho$ of rank greater than or equal to $2$,
the following identity holds:
$$
H^{\prime}(\Psi(P))
      =
\sum_{x \mbox{ {\rm \tiny coatom of} } P} \Psi([\hz,x]) .
$$
\end{lemma}

The next lemma gives the relation between
the toric Zaslavsky invariant $\Zt$ and
the map~$\lambda_{t}$.
\begin{lemma}
For a graded poset $P$ with $\ho$ of rank greater than or equal to $1$,
the following identity holds:
$$
     \lambda_{t}(\Psi(P)) = \Zt(P) \cdot (\av-\bv)^{\rho(P)-1}.
$$
\end{lemma}
\begin{proof}
When $P$ has rank $1$, both sides are equal to $1$.
For an $\ab$-monomial $v$ different from $1$, we have that
$\lambda_{t}(v) = \beta(H^{\prime}(v)) \cdot (\av-\bv)$.
Hence
\begin{eqnarray*}
\lambda_{t}(\Psi(P))
  & = &
\beta(H^{\prime}(\Psi(P))) \cdot (\av-\bv) \\
  & = &
\sum_{x \mbox{ {\rm \tiny coatom of} } P}
  \beta(\Psi([\hz,x])) \cdot (\av-\bv) \\
  & = &
(-1)^{\rho(P)} \cdot
\sum_{x \mbox{ {\rm \tiny coatom of} } P}
  \mu(\hz,x) \cdot (\av-\bv)^{\rho(P)-1} ,
\end{eqnarray*}
which concludes the proof.
\end{proof}

Define a sequence of functions
$\varphi_{t,k}\colon\zab\to\zab$ by $\varphi_{t,1}=\kappa$,
and for $k \geq 2$,
$$
\varphi_{t,k}(v)
  =
\sum_{v}
     \kappa(v_{(1)}) \cdot
     \bv \cdot \eta(v_{(2)}) \cdot
     \bv \cdot \eta(v_{(3)}) \cdot \bv \cdots
     \bv \cdot \eta(v_{(k-1)}) \cdot
     \bv \cdot \lambda_{t}(v_{(k)}).                      
$$
Finally, let $\varphi_{t}(v)$ be the
sum $\varphi_{t}(v) = \sum_{k\geq 1}\varphi_{t,k}(v)$.

\begin{theorem}
The $\ab$-index of the face poset $T_{t}$ 
of a toric arrangement is given by
$$
     \Psi(T_{t})^{*} = \varphi_{t}(\Psi(\Pzero)).
$$
\label{theorem_poset_varphi_toric}
\end{theorem}
\begin{proof}
The $\ab$-index of the poset $T_{t}$
is given by the sum 
$\Psi(T_{t}) = \sum_{c} |\zt^{-1}(c)| \cdot \wt(c)$.
Fix $k \geq 2$ and sum over all chains
$c = \{\hz = x_{0} < x_{1} < \cdots < x_{k} = \ho\}$
of length $k$.  We then have
\begin{eqnarray*}
  &   &
\sum_{c} |\zt^{-1}(c)| \cdot \wt(c) \\
  & = &
\sum_{c} \prod_{i=2}^{k-1} Z([x_{i-1}, x_{i}]) 
         \cdot \Zt([x_{k-1}, x_{k}])
  \cdot
      (\av-\bv)^{\rho(x_{0},x_{1})-1}
        \cdot
      \bv
        \cdots
      \bv
        \cdot
      (\av-\bv)^{\rho(x_{k-1},x_{k})-1} \\
  & = &
\sum_{c} 
   \kappa(\Psi([x_{0},x_{1}]))
       \cdot 
   \prod_{i=2}^{k-1} \left(\bv \cdot \eta(\Psi([x_{i-1}, x_{i}]))\right)
       \cdot 
   \bv
       \cdot 
   \lambda_{t}(\Psi([x_{k-1},x_{k}])) \\
  & = &
\sum_{w}
   \kappa(w_{(1)})
       \cdot 
   \prod_{i=2}^{k-1} \left(\bv \cdot \eta(w_{(i)})\right)
       \cdot 
   \bv
       \cdot 
   \lambda_{t}(w_{(k)}) \\
  & = &
\varphi_{t,k}(w)  ,  
\end{eqnarray*}
where we let $w$ denote the $\ab$-index of 
the augmented intersection poset $\Pzero$.
For $k=1$ we have that
$(\av-\bv)^{\rho(T_{t})-1} = \varphi_{t,1}(\Psi(\Pzero))$.
Summing over all $k \geq 1$, we obtain the result.
\end{proof}

\subsection{Evaluating the function $\varphi_{t}$}

\begin{proposition}
For an $\ab$-monomial $v$,
the following identity holds:
$$ 
        \varphi_{t}(v) 
    =
        \kappa(v)
      +
        \sum_{v} \varphi(v_{(1)}) \cdot \bv \cdot \lambda_{t}(v_{(2)}) .
$$
\label{proposition_varphi_t}
\end{proposition}
\begin{proof}
Using the coassociative identity
$\Delta^{k-1} = (\Delta^{k-2} \tensor \id) \circ \Delta$,
for $k \geq 2$ we have that
\begin{eqnarray*}
\varphi_{t,k}(v)
  & = &
\sum_{v}              \kappa(v_{(1)}) \cdot
                      \bv \cdot \eta(v_{(2)}) \cdot \bv \cdots
                      \bv \cdot \eta(v_{(k-1)}) \cdot
                      \bv \cdot \lambda_{t}(v_{(k)}) \\             
  & = &
\sum_{v}
\sum_{v_{(1)}}        \kappa(v_{(1,1)}) \cdot
                      \bv \cdot \eta(v_{(1,2)}) \cdot \bv \cdots
                      \bv \cdot \eta(v_{(1,k-1)}) \cdot
                      \bv \cdot \lambda_{t}(v_{(2)}) \\             
  & = &
\sum_{v}
                      \varphi_{k-1}(v_{(1)}) \cdot
                      \bv \cdot \lambda_{t}(v_{(2)}) .
\end{eqnarray*}
By  summing over all $k \geq 1$, the result follows.
\end{proof}

\begin{lemma}
Let $v$ be an $\ab$-monomial that begins with $\av$
and let $x$ be either $\av$ or $\bv$. Then
$$     \varphi_{t}(v \cdot \av \cdot x)
   = 
       \kappa(v \cdot \av \cdot x)
     + 
       {1}/{2} \cdot \omega(v \cdot \av\bv)  .  $$
\label{lemma_toric_varphi_I}
\end{lemma}
\begin{proof}
Using Proposition~\ref{proposition_varphi_t}
we have
\begin{eqnarray*}
       \varphi_{t}(v \cdot \av \cdot x)
  & = &
       \kappa(v \cdot \av \cdot x)
     +
       \varphi(v \cdot \av) \cdot \bv \cdot \lambda_{t}(1)
     +
       \varphi(v) \cdot \bv \cdot \lambda_{t}(x)  \\
  &   &
     + \sum_{v}
       \varphi(v_{(1)}) \cdot \bv \cdot
                 \lambda_{t}(v_{(2)} \cdot \bv \cdot x)  \\
  & = &
       \kappa(v \cdot \av \cdot x)
     +
       \varphi(v) \cdot \cv \cdot \bv
     +
       \varphi(v) \cdot \bv \cdot (\av-\bv) \\
  & = &
       \kappa(v \cdot \av \cdot x)
     +
       \omega(v) \cdot \dv \\
  & = &
       \kappa(v \cdot \av \cdot x)
     +
       1/2 \cdot \omega(v \cdot \av\bv) ,
\end{eqnarray*}
since $\lambda_{t}(v_{(2)} \cdot \bv \cdot x) = 0$.
\end{proof}

\begin{lemma}
Let $v$ be an $\ab$-monomial that begins with $\av$,
let $k$ be a positive integer 
and let $x$ be either $\av$ or $\bv$.
Then the following evaluation holds:
$$     \varphi_{t}(v \cdot \av \bv^{k} \cdot x)
   = 
       \kappa(v \cdot \av \bv^{k} \cdot x)
     + 
       {1}/{2} \cdot \omega(v \cdot \av\bv^{k+1})  .  $$
\label{lemma_toric_varphi_II}
\end{lemma}
\begin{proof_special}
Using Proposition~\ref{proposition_varphi_t}
we have
\begin{eqnarray}
       \varphi_{t}(v \cdot \av \bv^{k}\cdot x)
     -
       \kappa(v \cdot \av \bv^{k} \cdot x)
  & = &
       \varphi(v \cdot \av \bv^{k}) \cdot \bv \cdot \lambda_{t}(1)
     +
       \varphi(v \cdot \av) \cdot \bv \cdot \lambda_{t}(\bv^{k-1} \cdot x)
\nonumber \\
  &   &
     +
       \varphi(v) \cdot \bv \cdot \lambda_{t}(\bv^{k} \cdot x)
     +
       \sum_{i+j=k-2}
           \varphi(v \cdot \av \bv^{i+1}) \cdot \bv \cdot \lambda_{t}(\bv^{j} \cdot x)
\nonumber \\
  & = &
       \varphi(v)
     \cdot
       \left(
       \vphantom{\sum_{i+j=k-2}}
       2 \dv \cv^{k-1} \cdot \bv
     +
       \cv \cdot \bv \cdot (\av-\bv)^{k}
     +
       \bv \cdot (\av-\bv)^{k+1} 
       \right.
\nonumber \\
  &   &
       \left.
     +
       \sum_{i+j=k-2}
           2 \dv \cv^{i} \cdot \bv \cdot (\av-\bv)^{j+1}
       \right) .
\label{equation_messy}
\end{eqnarray}
In order to simplify this expression, consider the butterfly poset
of rank $k$. 
This is the poset consisting
of two rank $i$ elements, for $i = 1, \ldots, k-1$,
adjoined with a minimal and maximal element.
Each of the rank $i$ elements covers
the rank $i-1$ element(s) for $i = 1, \ldots, k-1$.
The butterfly poset
is the unique poset having the $\cd$-index $\cv^{k-1}$.
It is also Eulerian.
Applying~(\ref{equation_Stanley_recursion})
to the butterfly poset, we have
$$      \cv^{k-1}
    =
        (\av-\bv)^{k-1}
      +
        2 \cdot \sum_{i+j=k-2} \cv^{i} \cdot \bv \cdot (\av-\bv)^{j}
.  $$
Using this relation to simplify
equation~(\ref{equation_messy}), we obtain
\begin{eqnarray*}
\hspace*{35 mm}
       \varphi_{t}(v \cdot \av \bv^{k}\cdot x)
     -
       \kappa(v \cdot \av \bv^{k} \cdot x)
  & = &
       \varphi(v)
     \cdot
       \dv \cdot \cv^{k} \\
  & = &
       1/2
     \cdot
       \omega(v \cdot \av \bv^{k+1}) .
\hspace*{35 mm} \qed
\end{eqnarray*}
\end{proof_special}

By combining Lemmas~\ref{lemma_toric_varphi_I}
and~\ref{lemma_toric_varphi_II}, we have
the following proposition.
\begin{proposition}
For an $\ab$-monomial $v$ that begins with the letter $\av$,
the following holds:
$$     \varphi_{t}(v)
    =
       \kappa(v)
    +
       1/2
    \cdot
       \omega(H^{\prime}(v) \cdot \bv)  .  $$
\label{proposition_toric_varphi}
\end{proposition}

We now obtain the main result
for computing  the $\ab$-index
of the face poset of a toric arrangement.
\begin{theorem}
Let $\HH$ be a toric hyperplane arrangement on
the $n$-dimensional torus $T^{n}$
that subdivides the torus into a regular cell complex.
Then the $\ab$-index of the face poset $T_{t}$
can be computed from the 
$\ab$-index of the intersection poset $\PP$
as follows:
$$     \Psi(T_{t})
    =
        (\av-\bv)^{n+1}
    +
       \frac{1}{2}
     \cdot
       \omega(\av \cdot H^{\prime}(\Psi(\PP)) \cdot \bv)^{*}  .  $$
\label{theorem_toric}
\end{theorem}

Observe that in
Lemmas~\ref{lemma_toric_varphi_I} and~\ref{lemma_toric_varphi_II},
Proposition~\ref{proposition_toric_varphi}
and
Theorem~\ref{theorem_toric}
no
rational coefficients were introduced.
Only the $\ab$-monomial $\av^{n}$ is mapped
to a $\cd$-polynomial with an odd coefficient,
hence $1/2 \cdot \omega(v \cdot \bv)$
has all integer coefficients.

\addtocounter{theorem}{-20}
\begin{continuation}
{\rm
The flag $f$-vector of the intersection poset $\PP$ in
Example~\ref{example_toric_two}
is given by
$(f_{\emptyset},f_{1},f_{2},f_{12}) = (1,3,7,15)$,
the flag $h$-vector by
$(h_{\emptyset},h_{1},h_{2},h_{12}) = (1,2,6,6)$,
and so the  $\ab$-index is 
$\Psi(P) = \maa + 2 \cdot \mba + 6 \cdot \mab + 6 \cdot \mbb$.
Thus 
\begin{eqnarray*}
       \Psi(T_{t})
  & = &
        (\av-\bv)^{3}
    +
       1/2
     \cdot
       \omega(\av \cdot 
              H^{\prime}(\maa + 2 \cdot \mba + 6 \cdot \mab + 6 \cdot \mbb)
              \cdot \bv)^{*}  \\
  & = &
        (\av-\bv)^{3}
    +
       1/2
     \cdot
       \omega(\av \cdot (7 \cdot \ma + 8 \cdot \mb)
                \cdot \bv)^{*}  \\
  & = &
        (\av-\bv)^{3}
    +
       1/2
     \cdot
       \omega(7 \cdot \maab + 8 \cdot \mabb)^{*}  \\
  & = &
        (\av-\bv)^{3}
    +
       7 \cdot \mdc + 8 \cdot \mcd ,
\end{eqnarray*}
which agrees with the
calculation in Example~\ref{example_toric_two}.
}
\end{continuation}
\addtocounter{theorem}{19}

Theorem~\ref{theorem_toric}
gives a different approach 
from
Corollary~\ref{corollary_f_vector_I}
for determining 
the $f$-vector of $T_{t}$.
For notational ease, for positive integers $i$ and $j$,
let $[i,j] = \{ i, i+1, \ldots, j\}$
and
$[j] = \{1, \ldots, j\}$.

\begin{corollary}
The number of $i$-dimensional regions in the subdivision $T_{t}$
of the $n$-dimensional torus is given by the following
sum
of  flag $h$-vector entries from
the intersection poset $\PP$:
$$
      f_{i+1}(T_{t})
  =
      h_{[n-i,n]}(\PP)
   +
      h_{[n-i,n-1]}(\PP) 
   +
      h_{[n-i+1,n]}(\PP)
   +
      h_{[n-i+1,n-1]}(\PP)     ,
$$
for $1 \leq i \leq n-1$.
The number
of vertices is given by
$f_{1}(T_{t}) = 1 + h_{n}(\PP)$
and the number of
maximal regions
by
$f_{n+1}(T_{t})
 = h_{[n-1]}(\PP)
 + h_{[n]}(\PP)$.
\label{corollary_f_vector_II}
\end{corollary}
\begin{proof}
Let $\pair{\cdot}{\cdot}$ denote the inner product on $\zab$
defined by $\pair{u}{v} = \delta_{u,v}$
for two $\ab$-monomials $u$ and $v$.
For $1 \leq i \leq n-1$ we have
\begin{eqnarray*}
         f_{i+1}(T_{t})
  & = &
         1 + h_{i+1}(T_{t}) \\
  & = &
         1 + \pair{\av^{i} \bv \av^{n-i}}{\Psi(T_{t})} \\
  & = &
       \frac{1}{2}
     \cdot
       \pair{\av^{i} \bv \av^{n-i}}
       {\omega(\av \cdot H^{\prime}(\Psi(\PP)) \cdot \bv)^{*}}  \\
  & = &
       \frac{1}{2}
     \cdot
       [\cv^{i-1} \dv \cv^{n-i}]
       \omega(\av \cdot H^{\prime}(\Psi(\PP)) \cdot \bv)^{*}  
    +
       \frac{1}{2}
     \cdot
       [\cv^{i} \dv \cv^{n-i-1}]
       \omega(\av \cdot H^{\prime}(\Psi(\PP)) \cdot \bv)^{*}  \\
  & = &
       \pair{\av^{n-i} \cdot \av\bv \cdot \bv^{i-1}
              +
             \av^{n-i-1} \cdot \av\bv \cdot \bv^{i}}
            {\av \cdot H^{\prime}(\Psi(\PP)) \cdot \bv} \\
  & = &
       \pair{\av^{n-i-1} \cdot (\av+\bv) \cdot \bv^{i-1}}
            {H^{\prime}(\Psi(\PP))} \\
  & = &
       \pair{\av^{n-i-1} \cdot (\av+\bv) \cdot \bv^{i-1} \cdot (\av+\bv)}
            {\Psi(\PP)} .
\end{eqnarray*}
Expanding in terms of the flag $h$-vector the result follows.
The expressions for $f_{1}$ and $f_{n+1}$ are
obtained by similar calculations.
\end{proof}

The fact that Corollaries~\ref{corollary_f_vector_I}
and~\ref{corollary_f_vector_II}
are equivalent follows from
the coalgebra techniques
in Theorem~\ref{theorem_Ehrenborg_Readdy}.

\section{The complex of unbounded regions}
\label{section_affine}

\subsection{Zaslavsky and Bayer--Sturmfels}

The {\em unbounded Zaslavsky invariant} is defined by
$$    \Zub(P)
    =
      Z(P) - 2 \cdot Z_{b}(P)  .  $$
As the name suggests, the number of unbounded regions in 
a non-central
arrangement is given by this invariant.
By taking the difference of the
two statements in Theorem~\ref{theorem_Zaslavsky_poset} part (ii),
we immediately have the following result.
\begin{lemma}
For a non-central hyperplane arrangement $\HH$ the number of
unbounded regions is given by $\Zub(\LL)$,
where $\LL$
is the intersection lattice of the arrangement $\HH$.
\label{lemma_Zaslavsky_unbounded}
\end{lemma}

\begin{figure}
\begin{center}
\scalebox{.7}{\epsfig{file=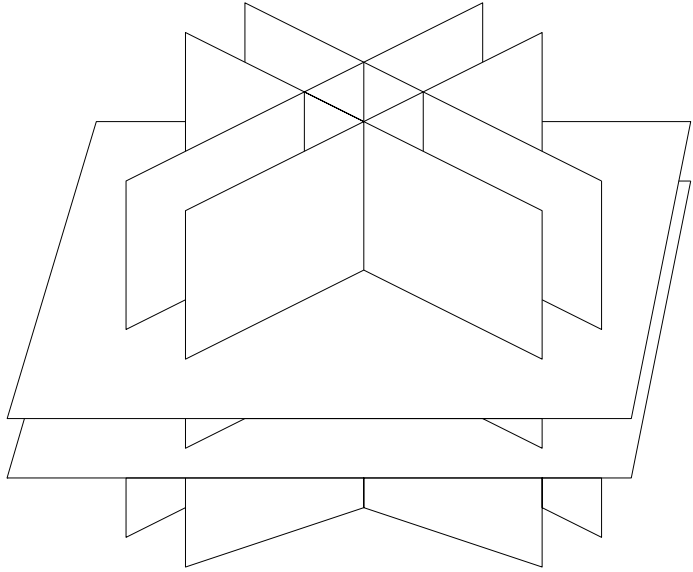}}
\end{center}
\caption{The non-central arrangement $x,y,z=0,1$.}
\label{figure_affine}
\end{figure}

Let $\HH$ be a non-central hyperplane arrangement 
in $\Rrr^n$
with intersection
lattice $\LL$ having the empty set $\emptyset$
as the maximal element.
Let $\Lub$ denote 
the {\em unbounded intersection lattice},
that is,
the subposet of the intersection lattice
consisting of all affine subspaces 
with the points (dimension zero affine subspaces) omitted
but with the empty set $\emptyset$ continuing to be the maximal element.
Equivalently, the poset
$\Lub$ is the rank-selected poset
$\LL([1, n-1])$,
that is,
the poset
$\LL$ with the coatoms removed.

Let $T$ be the face lattice of the arrangement
$\HH$ with the minimal element $\hz$ denoting the empty face
and the maximal element denoted by $\ho$.
Similarly, let $T_{ub}$ denote the set of all faces
in the face lattice~$T$ which are not bounded.
Observe that $T_{ub}$ includes the minimal and maximal elements of $T$
and that $T_{ub}$ is the face poset of
an $(n-1)$-dimensional sphere. 
Pick~$R$ large enough
so that all of the bounded faces
are strictly inside a ball of radius~$R$.
Intersect
the arrangement~$\HH$ with a sphere of radius $R$.
The resulting cell complex has face poset $T_{ub}$.
Our goal is to compute the $\cd$-index of $T_{ub}$ in terms
of the $\ab$-index of $\Lub$.

The collection of unbounded
faces of the arrangement $\HH$
forms a lower order ideal in the poset $T^{*}$.
Let~$Q$ be the subposet of $T^{*}$
consisting of this ideal with a  maximal element $\ho$
adjoined.
We define the rank of an element in $Q$ to be its rank
in the original poset $T^{*}$, that is, for $x \in Q$
let $\rho_{Q}(x) = \rho_{T^{*}}(x)$.
This rank convention will simplify the later arguments.
As posets $T_{ub}^{*}$ and $Q$ are isomorphic.
However, since their rank functions differ, their $\ab$-indexes 
satisfy
$\Psi(T_{ub})^{*} \cdot (\av-\bv) = \Psi(Q)$.

Restrict the zero map
$z : T^{*} \longrightarrow \Lzero$
to 
form the map $\zub : Q \longrightarrow \Lzero$.
Note that the map $\zub$ is order- and rank-preserving,
but not necessarily surjective.
As before we view the map $\zub$ as a map from the set of chains
of $Q$ to the set of chains of $\Lzero$.
We have the following theorem, which is analogous
to the Bayer--Sturmfels result
(Theorem~\ref{theorem_Bayer_Sturmfels}).
\begin{theorem}
Let $\HH$ be a non-central hyperplane arrangement
with intersection lattice $\LL$.
Let $c=\{\hz = x_0 < x_1 < \dots < x_k = \ho\}$ be a chain in
$\Lzero$ with $k \geq 2$.
Then the cardinality of the inverse image of the chain $c$
under $\zub$ is given by
$$      |\zub^{-1}(c)|
     =
        \prod_{i=2}^{k-1}
             Z([x_{i-1},x_{i}])
       \cdot
             \Zub([x_{k-1},x_{k}]) . $$
\label{theorem_Bayer_Sturmfels_unbounded}
\end{theorem}
\begin{proof}
We need to count the number of ways we can select a chain
$d = \{\hz = y_0 < y_1 < \dots < y_k = \ho\}$ in 
the poset of unbounded regions $Q$
such that $\zub(y_{i}) = x_{i}$.
The number of ways to select
the element~$y_{k-1}$ in $Q$
is the number of unbounded regions
in the arrangement restricted to the subspace $x_{k-1}$.
By Lemma~\ref{lemma_Zaslavsky_unbounded}
this can be done in $\Zub([x_{k-1},x_{k}])$ number of ways.
Since $y_{k-1}$ is an unbounded face of the arrangement
and all other elements in the chain $d$
contain the face $y_{k-1}$, the other elements must be unbounded.

The remainder of the proof is the same as
that of Theorem~\ref{theorem_Bayer_Sturmfels_toric}.
\end{proof}

\begin{corollary}
The flag $f$-vector entry
$f_{S}(T_{ub})$
is divisible by $2^{|S|}$
for any index set $S \subseteq \{1, \ldots, n\}$.
\end{corollary}
\begin{proof}
The proof is the same as 
Corollary~\ref{corollary_evenness}
with the extra observation that
the Zaslavsky invariant $\Zub$ is even.
\end{proof}

\subsection{The connection between posets and coalgebras}

Define $\lambda_{ub}$ by
$\lambda_{ub} = \eta - 2 \cdot \beta$.
By equations~(\ref{equation_beta}) 
and~(\ref{equation_eta}),
for a graded poset $P$ we have
$$  \lambda_{ub}(\Psi(P)) = \Zub(P) \cdot (\av-\bv)^{\rho(P)-1}  .  $$
Define a sequence of functions
$\varphi_{ub,k}\colon\zab\to\zab$ by $\varphi_{ub,1}=\kappa$ and for
$k>1$,
$$
\varphi_{ub,k}(v)
     =
\sum_{v}
       \kappa(v_{(1)}) \cdot
       \bv \cdot \eta(v_{(2)}) \cdot
       \bv \cdot \eta(v_{(3)}) \cdot \bv \cdots
       \bv \cdot \eta(v_{(k-1)}) \cdot
       \bv \cdot \lambda_{ub}(v_{(k)}).                      
$$
Finally, let $\varphi_{ub}(v)$ be the
sum $\varphi_{ub}(v) = \sum_{k\geq 1}\varphi_{ub,k}(v)$.

Similar to Theorem~\ref{theorem_poset_varphi_toric} we
have the next result.
The proof only differs in replacing
the map
$\zt : T_{t}^{*} \longrightarrow \Pzero$
with 
$\zub : Q \longrightarrow \Lzero$
and the invariant $\Zt$ by $\Zub$.
\begin{theorem}
The $\ab$-index of the poset $Q$ 
of unbounded regions of a
non-central 
arrangement is given by
$$
   \Psi(Q) = \varphi_{ub}(\Psi(\Lzero)).
$$
\label{theorem_poset_varphi_unbounded}
\end{theorem}

\subsection{Evaluating the function $\varphi_{ub}$}

In this subsection we
analyze the behavior of $\varphi_{ub}$.
\begin{lemma}
For any $\ab$-monomial $v$,
$$       \varphi_{ub}(v)
     =
         \varphi(v)
       -
         2 \cdot 
         \sum_v \varphi(v_{(1)}) \cdot \bv \cdot \beta(v_{(2)}) . $$
\label{lemma_phi_ub_recursion}
\end{lemma}
\begin{proof}
Using the coassociative identity
$\Delta^{k-1} = (\Delta^{k-2} \tensor \id) \circ \Delta$,
we have for $k \geq 2$
\begin{eqnarray*}
\varphi_{ub,k}(v)
  & = &
\varphi_{k}(v)
  -
2 \cdot 
\sum_{v}            \kappa(v_{(1)}) \cdot
                    \bv \cdot \eta(v_{(2)}) \cdot
                    \bv \cdots
                    \bv \cdot \eta(v_{(k-1)}) \cdot
                    \bv \cdot \beta(v_{(k)}) \\
  & = &
\varphi_{k}(v)
  -
2 \cdot 
\sum_{v}
\sum_{v_{(1)}}      \kappa(v_{(1,1)}) \cdot
                    \bv \cdot \eta(v_{(1,2)}) \cdot
                    \bv \cdots
                    \bv \cdot \eta(v_{(1,k-1)}) \cdot
                    \bv \cdot \beta(v_{(2)}) \\
  & = &
\varphi_{k}(v)
  -
2 \cdot 
\sum_{v}
\varphi_{k-1}(v_{(1)}) \cdot \bv \cdot \beta(v_{(2)}) .
\end{eqnarray*}
The result then follows by
summing over all $k \geq 2$ and adding
$\varphi_{ub,1}(v) = \kappa(v) = \varphi_{1}(v)$.
\end{proof}

\begin{lemma}
Let $v$ be an $\ab$-monomial.  Then 
$$   \varphi_{ub}(v \cdot \av)
   =
      \varphi(v) \cdot (\av - \bv) .   $$
\label{lemma_varphi_ub_I}
\end{lemma}
\begin{proof}
By Lemma~\ref{lemma_phi_ub_recursion}
and the Newtonian relation~(\ref{equation_Newtonian}) we have
$$   \varphi_{ub}(v \cdot \av)
 = 
     \varphi(v \cdot \av) 
   -
     2 \cdot
     \varphi(v) \cdot \bv \cdot \beta(1)
   -
     2 \cdot
     \sum_v \varphi(v_{(1)}) \cdot \bv \cdot \beta(v_{(2)} \cdot \av) .  $$
By equation~(\ref{equation_varphi_a})
$\varphi(v \cdot \av) = \varphi(v) \cdot \cv$.
The summation above is
zero because $\beta(v_{(2)} \cdot \av)$ is always zero.  Hence
$\varphi_{ub}(v \cdot \av) 
   =
  \varphi(v) \cdot (\cv - 2\bv)
   =
  \varphi(v) \cdot (\av - \bv)$.
\end{proof}

\begin{lemma}
Let $v$ be an $\ab$-monomial.  Then
\[
\varphi_{ub}(v \cdot \bv\bv) = \varphi_{ub}(v \cdot \bv) \cdot (\av - \bv).
\]
\label{lemma_varphi_ub_II}
\end{lemma}
\begin{proof}
Let $u = v \cdot \bv$. 
Applying
Lemma~\ref{lemma_phi_ub_recursion}
and the Newtonian relation~(\ref{equation_Newtonian})
to $u$ gives
\begin{eqnarray*}
\varphi_{ub}(u \cdot \bv) 
  & = &
\varphi(u \cdot \bv)
  - 
    2 \cdot
    \varphi(u) \cdot \bv \cdot \beta(1)
  - 
    2 \cdot
    \sum_{u} \varphi(u_{(1)}) \cdot \bv \cdot \beta(u_{(2)} \cdot \bv) \\
  & = &
\varphi(u) \cdot (\cv - 2\bv)
  - 
    2 \cdot
    \sum_{u} \varphi(u_{(1)}) \cdot \bv \cdot \beta(u_{(2)}) \cdot (\av - \bv) \\
  & = &
\left(\varphi(u) 
      - 
      2 \cdot
      \sum_{u} \varphi(u_{(1)}) \cdot \bv \cdot 
                 \beta(u_{(2)})
\right) \cdot (\av - \bv) \\
  & = &
\varphi_{ub}(u) \cdot (\av - \bv).
\end{eqnarray*}
Here we have used the two facts
$\varphi(u \cdot \bv) = \varphi(u) \cdot \cv$
and 
$\beta(u_{(2)} \cdot \bv) = \beta(u_{(2)}) \cdot (\av-\bv)$.
\end{proof}

\begin{lemma}
Let $v$ be an $\ab$-monomial.
Then $\varphi_{ub}(v \cdot \ab) = 0$.
\label{lemma_varphi_ub_III}
\end{lemma}
\begin{proof}
Directly we have
\begin{eqnarray*}
\varphi_{ub}(v \cdot \ab)
  & = &
  \varphi(v \cdot \ab)
   - 
     2 \cdot
     \varphi(v) \cdot \bv \cdot \beta(\bv)
   - 
     2 \cdot
     \varphi(v \cdot \av) \cdot \bv \cdot \beta(1) 
   - 
     2 \cdot
     \sum_v \varphi(v_{(1)}) \cdot \bv \cdot \beta(v_{(2)} \cdot \av \bv) \\
  & = &
  \varphi(v) \cdot 2\dv
- 
  2 \cdot \varphi(v) \cdot \bv \cdot (\av - \bv)
- 
  2 \cdot \varphi(v) \cdot \cv \bv \\
  & = &
  2 \cdot \varphi(v) \cdot (\dv - \bv(\av - \bv) - \cv\bv) \\
  & = &
0 ,
\end{eqnarray*}
where we have used the facts
$\varphi(v \cdot \av\bv) = \varphi(v) \cdot 2\dv$
and 
$\beta(v_{(2)} \cdot \av\bv) = 0$.
\end{proof}

The previous three lemmas enable us to determine $\varphi_{ub}$.
In order to obtain more compact notation, 
define a map $r\colon\zab\to\zab$ by $r(1)=0$,
$r(v \cdot \av)=v$,
and
$r(v \cdot \bv)=0$.
By using the chain definition of the $\ab$-index,
it is straightforward to see that
$\Psi(\Lub) = r(\Psi(\LL))$.

\begin{proposition}
Let $w$ be an $\ab$-polynomial homogeneous of degree greater than zero.
Then
$$    \varphi_{ub}(\av \cdot w)
    =
      \omega(\av \cdot r(w)) \cdot (\av - \bv)   .  $$
\end{proposition}
\begin{proof}
The case
$w = v \cdot \av$ 
follows from Lemma~\ref{lemma_varphi_ub_I}.
The remaining  case is $w = v \cdot \bv$.
Note that $\av \cdot v \cdot \bv$ can be factored as
$u \cdot \av\bv \cdot \bv^{k}$ for a monomial $u$.
Hence
$\varphi_{ub}(u \cdot \av\bv \cdot \bv^{k})
  =
 \varphi_{ub}(u \cdot \av\bv) \cdot (\av-\bv)^{k}
  =
 0$
by Lemmas~\ref{lemma_varphi_ub_II}
and~\ref{lemma_varphi_ub_III}.
\end{proof}

We combine all of these results to 
conclude that the $\cd$-index of
the poset of unbounded regions~$T_{ub}$ can be computed
in terms of the $\ab$-index of the unbounded intersection lattice
$\Lub$.

\begin{theorem}
Let $\HH$ be a non-central hyperplane arrangement with
the unbounded intersection
lattice $\Lub$ and poset of unbounded regions $T_{ub}$.
Then the $\ab$-index of $T_{ub}$ is given by
$$    \Psi(T_{ub})
    =
      \omega(\av \cdot \Psi(\Lub))^{*} . $$
\label{theorem_unbounded}
\end{theorem}
\begin{proof}
We have that
\begin{eqnarray*}
\Psi(T_{ub})^{*} \cdot (\av-\bv)
  & = &
\Psi(Q) \\
  & = &
\varphi_{ub}(\av \cdot \Psi(\LL)) \\
  & = &
\omega(\av \cdot r(\Psi(\LL))) \cdot (\av-\bv) \\
  & = &
\omega(\av \cdot \Psi(\Lub)) \cdot (\av-\bv) .
\end{eqnarray*}
By cancelling $\av-\bv$ on both sides of this identity,
the result follows.
\end{proof}

\begin{figure}
\begin{center}
\scalebox{.3}{\epsfig{file=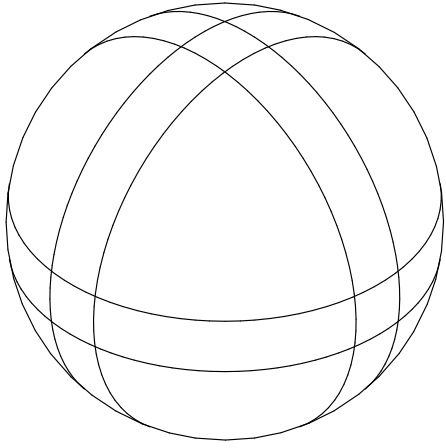}}
\end{center}
\caption{The spherical subdivision
obtained from the non-central arrangement $x,y,z=0,1$.}
\label{figure_spherical}
\end{figure}

\begin{example}
{\rm
Consider the non-central hyperplane arrangement
consisting of the six hyperplanes
$x = 0,1$, 
$y = 0,1$ and
$z = 0,1$.
See Figure~\ref{figure_affine}.
After intersecting this arrangement
with a sphere of large enough radius
we obtain the cell complex in
Figure~\ref{figure_spherical}.
The polytopal realization of this complex
is known as the rhombicuboctahedron.
The dual of the face lattice of this spherical complex is
not realized by a zonotope. However, one can view
the dual lattice as the face lattice
of a $2 \times 2 \times 2$ pile of cubes.

The intersection lattice $\LL$ is the
face lattice of the three-dimensional crosspolytope,
in other words,
the octahedron.
Hence the lattice of unbounded intersections
$\Lub$
has the flag $f$-vector
$(f_{\emptyset},f_{1},f_{2},f_{12}) = (1,6,12,24)$
and the flag $h$-vector 
$(h_{\emptyset},h_{1},h_{2},h_{12}) = (1,5,11,7)$.
The $\ab$-index is given by
$\Psi(\Lub) = \maa + 5 \cdot \mba + 11 \cdot \mab + 7 \cdot \mbb$.
Hence the $\cd$-index of $T_{ub}$ is
\begin{eqnarray*}
\Psi(T_{ub})
  & = &
      \omega(\maaa + 5 \cdot \maba + 11 \cdot \maab + 7 \cdot \mabb)^{*} \\
  & = &
      \mccc + 22 \cdot \mdc + 24 \cdot \mcd  .
\end{eqnarray*}
}
\end{example}

\section{Concluding remarks}

For regular subdivisions of manifolds there is now
a plethora of questions to ask.
\begin{itemize}
\item[(i)] What is the right analogue of a regular subdivision
in order that it  be polytopal? 
Can flag $f$-vectors be classified for
polytopal subdivisions?
\item[(ii)] Is there a Kalai convolution for manifolds
that will generate more inequalities for flag $f$-vectors?
\cite{Kalai}
\item[(iii)] Is there a lifting technique that will yield
             more inequalities for higher dimensional manifolds?
\cite{Ehrenborg_lifting}
\item[(iv)] Are there minimization inequalities for the
$\cd$-coefficients in the polynomial $\Psi$?
As a first step, can one prove the non-negativity of $\Psi$?
\cite{Billera_Ehrenborg,Ehrenborg_Karu}
\item[(v)] Is there an extension of the toric $g$-inequalities
to manifolds?
\cite{Bayer_Ehrenborg,Kalai_g,Karu,Stanley_h}
\item[(vi)] Can the coefficients for $\Psi$ be minimized for regular toric
arrangements as was done in the case of central hyperplane arrangements?
\cite{Billera_Ehrenborg_Readdy_om}
\end{itemize}

The most straightforward manifold to study is
$n$-dimensional projective space $P^{n}$. We offer the following
result in obtaining the $\ab$-index of subdivisions of $P^{n}$.
\begin{theorem}
Let $\Omega$ be a centrally symmetric regular subdivision of
the $n$-dimensional sphere $S^{n}$. Assume that when 
antipodal points of the sphere are identified,
a regular subdivision
$\Omega^{\prime}$ of the projective space $P^{n}$
is obtained.
Then the $\ab$-index of $\Omega^{\prime}$ is given by
$$    \Psi(\Omega^{\prime})
   =
      \frac{\cv^{n+1} + (\av-\bv)^{n+1}}{2}
   +
      \frac{\Phi}{2}   ,   $$
where the $\cd$-index of $\Omega$ is
$\Psi(\Omega) = \cv^{n+1} + \Phi$.
\end{theorem}
\begin{proof}
Each chain $c = \{\hz = x_{0} < x_{1} < \cdots < x_{k} = \ho\}$ with
$k \geq 2$ in $\Omega^{\prime}$ corresponds to two chains in~$\Omega$
with the same weight $\wt(c)$. The chain
$c = \{\hz = x_{0} < x_{1} = \ho\}$ corresponds to exactly
one chain in~$\Omega$ and has weight $(\av-\bv)^{n+1}$.
Hence
$\Psi(\Omega) = 2 \cdot \Psi(\Omega^{\prime}) - (\av-\bv)^{n+1}$,
proving the result.
\end{proof}

The results in this paper have been stated
for hyperplane arrangements. In true generality
one could
work with the underlying oriented matroid,
especially  since there are nonrealizable ones 
such as the non-Pappus oriented matroid.
All of these can be  represented as pseudo-hyperplane
arrangements. However, we have chosen to work with
hyperplane arrangements in order not to lose
the geometric intuition.

Poset transformations related to the $\omega$ map appear
in~\cite{Ehrenborg_r_Birkhoff,Ehrenborg_Readdy_Tchebyshev,Hsiao}.
Are there toric or affine analogues of these poset transforms?

Another way to encode the flag $f$-vector data of a poset
is to use the quasisymmetric function of a
poset~\cite{Ehrenborg_Hopf}.
In this language the $\omega$ map is translated to
Stembridge's $\vartheta$ map;
see~\cite{Billera_Hsiao_van_Willigenburg,Stembridge}.
Would the results of
Theorems~\ref{theorem_toric}
and~\ref{theorem_unbounded}
be appealing in the quasisymmetric function viewpoint?

Richard Stanley has asked if the coefficients of the toric
characteristic polynomial are alternating. If so, is there any
combinatorial interpretation of the absolute values of the
coefficients.

A far reaching generalization of Zaslavsky's results for hyperplane
arrangements is by Goresky and MacPherson~\cite{Goresky_MacPherson}.
Their results determine the cohomology groups of the
complement of a complex hyperplane arrangement. 
For a toric analogue of the Goresky--MacPherson results,
see work of De~Concini and Procesi~\cite{De_Concini_Procesi}.
For algebraic considerations of toric arrangements,
see~\cite{Douglass,Macmeikan_I, Macmeikan_II, Macmeikan_III}.

Greene and Zaslavsky~\cite{Greene_Zaslavsky}
also give an interpretation of the derivative of the chromatic
polynomial of a graph evaluated at $1$. Given an edge $ij$ of a graph $G$,
the number of acyclic orientations of the graph
with unique source at $i$ and
unique sink at $j$ is given by
$(-1)^{n} \cdot \left.\frac{d}{dt} \chi(G;t)\right|_{t=1}$.
Is there a geometric proof of this fact
analogous to the proof of 
Theorem~\ref{theorem_Greene_Zaslavsky}?

In Section~\ref{section_toric} we restricted ourselves
to studying arrangements that cut the torus into
regular cell complexes. In a future paper~\cite{Ehrenborg_Slone},
two of the authors are developing the notion
of a $\cd$-index for non-regular cell complexes.

\section*{Acknowledgements}

The authors thank
the MIT Mathematics Department where
this research was carried out.
We also thank the
referee for his careful comments
and for bringing the 1977 paper of Zaslavsky
to our attention,
Andrew Klapper for his suggestions to improve the exposition,
and Tricia Hersh for suggesting that we look at
graphical arrangements on the torus.
The first and third authors were partially
supported by
National Security Agency grant H98230-06-1-0072.

\newcommand{\journal}[6]{{\sc #1,} #2, {\it #3} {\bf #4} (#5), #6.}
\newcommand{\book}[4]{{\sc #1,} ``#2,'' #3, #4.}
\newcommand{\bookf}[5]{{\sc #1,} ``#2,'' #3, #4, #5.}
\newcommand{\books}[6]{{\sc #1,} ``#2,'' #3, #4, #5, #6.}
\newcommand{\collection}[6]{{\sc #1,}  #2, #3, in {\it #4}, #5, #6.}
\newcommand{\thesis}[4]{{\sc #1,} ``#2,'' Doctoral dissertation, #3, #4.}
\newcommand{\springer}[4]{{\sc #1,} ``#2,'' Lecture Notes in Math.,
                          Vol.\ #3, Springer-Verlag, Berlin, #4.}
\newcommand{\preprint}[3]{{\sc #1,} #2, preprint #3.}
\newcommand{\preparation}[2]{{\sc #1,} #2, in preparation.}
\newcommand{\appear}[3]{{\sc #1,} #2, to appear in {\it #3}}
\newcommand{\submitted}[3]{{\sc #1,} #2, submitted to {\it #3}}
\newcommand{\JCTA}{J.\ Combin.\ Theory Ser.\ A}
\newcommand{\AdvancesinMathematics}{Adv.\ Math.}
\newcommand{\JournalofAlgebraicCombinatorics}{J.\ Algebraic Combin.}

\newcommand{\communication}[1]{{\sc #1,} personal communication.}


\bigskip

{\em R.\ Ehrenborg, M.\ Readdy and M.\ Slone,
Department of Mathematics,
University of Kentucky,
Lexington, KY 40506-0027,}
\{{\tt jrge},{\tt readdy},{\tt mslone}\}{\tt @ms.uky.edu}


{\small

\begin{thebibliography}{99}

\bibitem{Athanasiadis}
\journal{C.\ A.\ Athanasiadis}
        {Characteristic polynomials of subspace arrangements
        and finite fields}
        {\AdvancesinMathematics}
        {122}{1996}{193--233}

\bibitem{Athanasiadis_II}
\journal{C.\ A.\ Athanasiadis}
        {Extended Linial hyperplane arrangements for root systems
         and a conjecture of Postnikov and Stanley}
        {J.\ Algebraic Combin.}
        {10}{1999}{207--225}

\bibitem{Bayer_Billera}
\journal{M.\ Bayer and L.\ Billera}
        {Generalized Dehn-Sommerville relations for polytopes,
         spheres and Eulerian partially ordered sets}
        {Invent.\ Math.}
        {79}{1985}{143--157}

\bibitem{Bayer_Ehrenborg}
\journal{M.\ Bayer and R.\ Ehrenborg}
        {The toric $h$-vectors of partially ordered sets}
        {Trans.\ Amer.\ Math.\ Soc.}
        {352}{2000}{4515--4531}

\bibitem{Bayer_Klapper}
\journal{M.\ Bayer and A.\ Klapper}
        {A new index for polytopes}
        {Discrete Comput.\ Geom.}
        {6}{1991}{33--47}

\bibitem{Bayer_Sturmfels}
\journal{M.\ Bayer and B.\ Sturmfels}
        {Lawrence polytopes}
        {Canad.\ J.\ Math.}
        {42}{1990}{62--79}

\bibitem{Billera_Ehrenborg}
\journal{L.\ J.\ Billera and R.\ Ehrenborg}
        {Monotonicity of the cd-index for polytopes}
        {Math.\ Z.}
        {233}{2000}{421--441}

\bibitem{Billera_Ehrenborg_Readdy_om}
\journal{L.\ J.\ Billera, R.\ Ehrenborg, and M.\ Readdy}
        {The {\ctd}-index of oriented matroids}
        {J.\ Combin.\ Theory Ser.\ A}
        {80}{1997}{79--105}

\bibitem{Billera_Hsiao_van_Willigenburg}
\journal{L.\ J.\ Billera, S.\ K.\ Hsiao and S.\ van Willigenburg}
        {Peak quasisymmetric functions and Eulerian enumeration}
        {\AdvancesinMathematics}
        {176}{2003}{248--276}

\bibitem{Bjorner_topological_methods}
{\sc A.\ Bj\"orner,}
Topological methods.
Handbook of combinatorics,
Vol.\ 2, 1819--1872,
Elsevier, Amsterdam, 1995. 

\bibitem{Chen}
\journal{B.\ Chen}
        {On characteristic polynomials of
         subspace arrangements}
        {\JCTA}
        {90}{2000}{347--352}

\bibitem{De_Concini_Procesi}
\journal{C.\ De Concini and C.\ Procesi}
        {On the geometry of toric arrangements}
        {Transform.\ Groups}
        {10}{2005}{387--422}

\bibitem{Douglass}
\journal{J.\ M.\ Douglass}
        {Toral arrangements and hyperplane arrangements}
        {Rocky Mountain J.\ Math.}
        {28}{1998}{939--956}



\bibitem{Ehrenborg_Hopf}
\journal{R.\ Ehrenborg}
        {On posets and Hopf algebras}
        {\AdvancesinMathematics}
        {119}{1996}{1--25}

\bibitem{Ehrenborg_k-Eulerian}
\journal{R.\ Ehrenborg}
        {$k$-Eulerian posets}
        {Order.}
        {18}{2001}{227--236}

\bibitem{Ehrenborg_r_Birkhoff}
\submitted{R.\ Ehrenborg}
          {The $r$-signed Birkhoff transform}
          {Trans.\ Amer.\ Math.\ Soc.}

\bibitem{Ehrenborg_lifting}
\journal{R.\ Ehrenborg}
        {Lifting inequalities for polytopes}
        {Adv.\ Math.}
        {193}{2005}{205--222}

\bibitem{Ehrenborg_Karu}
\journal{R.\ Ehrenborg and K.\ Karu}
        {Decomposition theorem for the $\cd$-index of Gorenstein* posets}
        {J.\ Algebraic Combin.}
        {26}{2007}{225--251}

\bibitem{Ehrenborg_Readdy_c}
\journal{R.\ Ehrenborg and M.\ Readdy}
        {Coproducts and the $\cd$-index}
        {\JournalofAlgebraicCombinatorics}
        {8}{1998}{273--299}

\bibitem{Ehrenborg_Readdy_valuation_1}
\journal{R.\ Ehrenborg and M.\ Readdy}
        {On valuations, the characteristic polynomial and
         complex subspace arrangements}
        {Adv.\ Math.}
        {134}{1998}{32--42}

\bibitem{Ehrenborg_Readdy_valuation_2}
\journal{R.\ Ehrenborg and M.\ Readdy}
        {The Dowling transform of subspace arrangements}
        {\JCTA}
        {91}{2000}{322--333}

\bibitem{Ehrenborg_Readdy_homology}
\journal{R.\ Ehrenborg and M.\ Readdy}
        {Homology of Newtonian coalgebras}
        {European J.\ Combin.}
        {23}{2002}{919--927}


\bibitem{Ehrenborg_Readdy_Tchebyshev}
\appear{R.\ Ehrenborg and M.\ Readdy}
       {The Tchebyshev transforms of the first and second kind}
       {Ann.\ Comb.}


\bibitem{Ehrenborg_Slone}
\preparation{R.\ Ehrenborg and M.\ Slone}
            {The $\cd$-index of non-regular cell complexes}



\bibitem{Goresky_MacPherson}
  {\sc M.\ Goresky and R.\ MacPherson,}
  ``Stratified Morse theory''
  Ergebnisse der Mathematik und ihrer Grenzgebiete (3)
  [Results in Mathematics and Related Areas (3)],
  14.\ Springer-Verlag, Berlin, 1988.


\bibitem{Greene_Zaslavsky}
\journal{C.\ Greene and T.\ Zaslavsky}
        {On the interpretation of Whitney numbers
         through arrangements of hyperplanes, zonotopes,
         non-Radon partitions, and orientations of graphs}
        {Trans.\ Amer.\ Math.\ Soc.}
        {280}{1983}{97--126}


\bibitem{Groemer}
\journal{H.\ Groemer}
        {On the extension of additive functionals on classes of convex sets}
        {Pacific J.\ Math.}
        {75}{1978}{397--410}


\bibitem{Grunbaum}
\books{B.\ Gr\"unbaum}
      {Convex polytopes}
      {second edition}
      {Springer-Verlag}{New York}{2003}


\bibitem{Hsiao}
\journal{S.\ K.\ Hsiao}
        {A signed analog of the Birkhoff transform}
        {\JCTA}
        {113}{2006}{251--272}


\bibitem{Joni_Rota}
\journal{S.\ A.\ Joni and G.-C.\ Rota}
        {Coalgebras and bialgebras in combinatorics}
        {Stud.\ Appl.\ Math.}
        {61}{1979}{93--139}


\bibitem{Jozefiak_Sagan}
\journal{T.\ J\'ozefiak and B.\ Sagan}
        {Basic derivations for subarrangements of Coxeter arrangements}
        {J.\ Algebraic Combin.}
        {2}{1993}{291--320}

\bibitem{Kalai_g}
\journal{G.\ Kalai}
        {Rigidity and the lower bound theorem. I}
        {Invent.\ Math.}
        {88}{1987}{125--151}

\bibitem{Kalai}
\journal{G.\ Kalai}
	{A new basis of polytopes}
	{\JCTA}
	{49}{1988}{191--209}

\bibitem{Karu}
\journal{K.\ Karu}
        {Hard Lefschetz theorem for nonrational polytopes}
        {Invent.\ Math.}
        {157}{2004}{419--447}

\bibitem{Klain_Rota}
\book{D.\ Klain and G.-C.\ Rota}
     {Introduction to geometric probability}
     {Lezioni Lincee, Cambridge University Press, Cambridge}
     {1997}



\bibitem{Macmeikan_I}
\journal{C.\ Macmeikan}
        {The Poincar\'e polynomial of an mp arrangement}
        {Proc.\ Amer.\ Math.\ Soc.}
        {132}{2004}{1575--1580}

\bibitem{Macmeikan_II}
\journal{C.\ Macmeikan}
        {Modules of derivations for toral arrangements}
        {Indag.\ Math. (N.S.)}
        {15}{2004}{257--267}

\bibitem{Macmeikan_III}
\collection{C.\ Macmeikan}
        {Toral arrangements.  The COE Seminar on
         Mathematical Sciences 2004}
         {37--54}
        {Sem.\ Math.\ Sci., 31}
        {Keio Univ., Yokohama}
        {2004}


\bibitem{Novik_Postnikov_Sturmfels}
\journal{I.\ Novik, A.\ Postnikov and B.\ Sturmfels}
        {Syzygies of oriented matroids}
        {Duke Math.\ J.}
        {111}{2002}{287--317}


\bibitem{Stanley_acyclic}
\journal{R.\ Stanley}
        {Acyclic orientations of graphs}
        {Discrete Math.}
        {5}{1973}{171--178}


\bibitem{Stanley_EC_I}
\book{R.\ P.\ Stanley}
     {Enumerative Combinatorics, Vol. I}
     {Wadsworth and Brooks/Cole, Pacific Grove}
     {1986}


\bibitem{Stanley_h}
{\sc R.\ P.\ Stanley,}
Generalized $h$-vectors, intersection cohomology of toric varieties,
and related results.
Commutative algebra and combinatorics (Kyoto, 1985),
187--213, Adv. Stud. Pure Math., 11, North-Holland, Amsterdam, 1987.


\bibitem{Stanley_d}
\journal{R.\ P.\ Stanley}
        {Flag $f$-vectors and the $cd$-index}
        {Math.\ Z.}
        {216}{1994}{483--499}

\bibitem{Stembridge}
\journal{J.\ Stembridge}
        {Enriched $P$-partitions}
        {Trans.\ Amer.\ Math.\ Soc.}
        {349}{1997}{763--788}

\bibitem{Swartz}
\preprint{E.\ Swartz}
         {Face enumeration -- from spheres to manifolds}
         {2007}

\bibitem{Sweedler}
\book{M.\ Sweedler}
     {Hopf Algebras}
     {Benjamin, New York}
     {1969}

\bibitem{Zaslavsky}
        {\sc T.\ Zaslavsky,}
	{Facing up to arrangements:  face count formulas for
	partitions of space by hyperplanes,}
	{\it Mem.\ Amer.\ Math.\ Soc.}
	{\bf 154}
        {(1975).}


\bibitem{Zaslavsky_paper}
\journal{T.\ Zaslavsky}
        {A combinatorial analysis of topological dissections}
        {\AdvancesinMathematics}
        {25}{1977}{267--285}


\end{thebibliography}

}
\end{document}